\documentclass{article}

\usepackage{amsmath,amsthm,amssymb,mathtools}  % These are to write mathematics
\usepackage{dsfont}				% Símbolos para os conjuntos R, N, etc.
\usepackage{enumerate}
\usepackage{graphicx}			% Para incluir figuras

\usepackage[T1]{fontenc}		% acentos em outras linguas (precisa para a ref do Gorski)

\usepackage{bm}					% bold greek letters in math
\usepackage{cmll}				% \Perp, \simperp
\usepackage{scalerel}			% \scaleobj usado para mudar tamanho do \boxtimes

\usepackage[nameinlink]{cleveref} 	% comando \cref e \Cref ao inves de \autoref  (pacote deve ser carregado depois do hyperref)

\usepackage[latin1]{inputenc}                       % Para poder colocar acentos diretamente
\usepackage[ disable,
			colorinlistoftodos,textsize=scriptsize]{todonotes}		% anotações de margem, lista de anotações
\newcommand{\SELF}[1]{\todo[color=green!50]{#1}} % note to myself
\newtheorem{theorem}{Theorem}[section]
\newtheorem{proposition}[theorem]{Proposition}
\newtheorem{corollary}[theorem]{Corollary}
\newtheorem{lemma}[theorem]{Lemma}

\theoremstyle{remark}
\newtheorem{example}[theorem]{Example}

\theoremstyle{definition}
\newtheorem{definition}[theorem]{Definition}

\newcommand\inner[1] 		{\langle #1 \rangle}		% inner product
\def\Proj		{P}					%{\mathrm{Proj}}
\def\ii			{\mathbf{i}}		% multi-index i
\def\jj			{\mathbf{j}}		% multi-index j
\def\II			{\mathcal{I}}		% multi-indices set
\def\I			{I}					% principal bivector(s)
\def\btheta		{\Phi} %{\bm{\theta}}		% angle bivector 

\def\lcontr		{\lrcorner}			% left contraction
\def\rcontr		{\llcorner}			% right contraction
\def\glcontr	{\rfloor}			% left contraction (geometric algebra)
\def\grcontr	{\lfloor}			% right contraction (geometric algebra)

\def\N{\mathds{N}}

\def\R{\mathds{R}}
\def\C{\mathds{C}}
\def\P{P} %{\mathcal{P}}
\def\PP{\mathds{P}}

\def\im	 				{\mathrm{i}}
\def\oo					{\infty}

\def\pperp			{\simperp}

\DeclareMathOperator{\Span}{span}

\DeclareMathOperator{\vol}{vol}

\DeclareMathOperator{\grade}{grade}

\def\Pythagorean	{Pyth\-a\-go\-re\-an}

\newcommand\acom{\mathrel{\scaleobj{0.8}{\boxtimes}}}

\begin{document}

\title{Blade Products and the Angle Bivector of Subspaces}

\author{Andr\'e L. G. Mandolesi 
	\thanks{Instituto de Matemática e Estatística, Universidade Federal da Bahia, Av. Adhemar de Barros s/n, 40170-110, Salvador - BA, Brazil. ORCID 0000-0002-5329-7034. E-mail: \texttt{andre.mandolesi@ufba.br}}}

\date{\today}

\maketitle

\begin{abstract}
	Principal angles are used to define an angle bivector of subspaces, which fully describes their relative inclination.
	Its exponential is related to the Clifford geometric product of blades, gives rotors connecting subspaces via minimal geodesics in Grassmannians, and decomposes giving Plücker coordinates, projection factors and angles with various subspaces.
	This leads to new geometric interpretations for this product and its properties, and to formulas relating other blade products (scalar, inner, outer, etc., including those of Grassmann algebra) to angles between subspaces. Contractions are linked to an asymmetric angle, while commutators and anticommutators involve hyperbolic functions of the angle bivector, shedding new light on their properties.

	\vspace{.5em}
	\noindent
	{\bf Keywords:}  Clifford algebra, geometric algebra, Grassmann algebra, blade product, angle between subspaces, angle bivector, asymmetric angle

	\vspace{3pt}
	
	\noindent
	{\bf MSC:} 	15A66,15A75	% Clifford algebras, spinors % Exterior algebra, Grassmann algebras
\end{abstract}

\maketitle

\section{Introduction}

Much of the usefulness of dot and cross products of vectors comes from the formulas relating them to angles. 
Clifford geometric algebra has similar results for the scalar product and contraction of blades \cite{Dorst2007}, but the relation between other products and angles has been mostly ignored.
An exception is Hitzer's formula \cite{Hitzer2010a} relating the Clifford geometric product of blades to principal angles \cite{Jordan1875,Wedin1983}, but it is too complex to really help us understand the product, and its main purpose was to compute the angles. 

A difficulty in relating blade products to angles is that blades can represent high dimensional subspaces, for which there are various angle concepts. 
Measuring the separation of subspaces is useful in geometry, linear algebra, functional analysis, statistics, computer vision, data mining, etc., but a full description of their relative inclination requires a list of principal angles, which can be cumbersome. So, depending on the purpose, one usually takes the smallest, largest, or some function of principal angles describing whatever relation between the subspaces is most relevant \cite{Deutsch1995,Galantai2008,Ye2016}.
This can lead to misunderstandings, as distinct concepts are often called \emph{the} angle between subspaces, despite each having its own properties and limitations.

The geometric algebra literature has some (not fully equivalent) definitions for the angle between blades or subspaces \cite{Dorst2007,Hestenes1984clifford,Hitzer2010a}, but it does not seem to be well understood, being usually described only in simple cases, by comparison with usual angles in $\R^3$ (sometimes erroneously, as we show).
This angle is similar to the angles introduced in \cite{Gluck1967,Gunawan2005,Jiang1996,Wedin1983} and which measure separation of subspaces in terms of how volumes contract when orthogonally projected between them (so, projection factors \cite{Mandolesi_Pythagorean}).

These angles have been unified and generalized into an asymmetric angle of subspaces \cite{Mandolesi_Grassmann}, which has better properties. Its unusual asymmetry for subspaces of distinct dimensions turns out to be an advantage, making its use more efficient and leading to more general results. This angle has been linked to the Grassmann algebra and the geometry of Grassmannians, and here we show it is also deeply connected with the Clifford algebra.

Instead of a scalar angle, Hawidi \cite{Hawidi1966} proposed an angle operator carrying all data about the relative inclination of subspaces, but it was never widely adopted.
Fortunately, geometric algebra has a better way to store such information:
\SELF{even some absolute positioning info, as a rotation of the pair $(V,W)$ requires a rotation of $\btheta$}
describing relative inclination is akin to telling how to rotate one subspace into another, which calls for the use of rotors and bivectors.

Using principal angles and vectors, we define angle bivectors whose exponentials give rotors connecting subspaces through minimal geodesics in Grassmannians. When properly decomposed, such exponentials give angles and projection factors for various subspaces, and also Plücker coordinates.
Hitzer's formula is turned into a simple relation between the Clifford product of blades and the exponential of the angle bivector.

Clifford algebra gives strong results with such ease that the geometry behind them can  be missed just as easily. Our results provide new geometric interpretations for the Clifford product and some of its well known algebraic properties.
For example, Plücker coordinates stored in the product allow the invertibility of non-null blades, the relation $\|AB\|=\|A\|\|B\|$ for blades is linked to a \Pythagorean\ theorem for volumes \cite{Mandolesi_Pythagorean}, and the dualities \cite{Dorst2007} $(AB)^*=AB^*$, $(A\wedge B)^* = A\glcontr B^*$ and $(A\glcontr B)^* = A\wedge B^*$ reflect a symmetry in the exponentials of angle bivectors.

The formula relating the Clifford product to the angle bivector yields others for the various geometric algebra products of blades: Lounesto's asymmetric contractions \cite{Lounesto1993} are linked to the asymmetric angle, the outer product to a complementary angle, and the scalar product, Hestenes inner product and Dorst's dot product \cite{Dorst2002} to symmetrized angles.
Blade commutators and anticommutators are related to hyperbolic functions of the angle bivector. 
We also give formulas for Grassmann algebra products, which are used in \cite{Mandolesi_Grassmann} to obtain formulas for computing asymmetric angles and identities that lead to properties of projection factors \cite{Mandolesi_Pythagorean}.

The symmetrized angle related to Hestenes inner product has worse properties than the asymmetric one, and this supports Dorst's case for the use of contractions instead of that product.
Their asymmetry is directly linked to that of asymmetric angles, and in both cases it leads to better results with simpler proofs.

\Cref{sc:preliminaries} reviews some concepts and results.
\Cref{sc:Angle bivector} introduces the angle bivector and studies its exponential.
\Cref{sc:geometric product} relates the Clifford product to the angle bivector, and interprets geometrically some of its properties.
\Cref{sc:Other products} relates other geometric algebra products to angles.
\Cref{sc:Grassmann products} does the same for Grassmann algebra products, and requires only \Cref{sc:preliminaries}.
\Cref{sc:Hyperbolic Functions} develops some properties of hyperbolic functions of multivectors.

\section{Preliminaries}\label{sc:preliminaries}

In this article\footnote{Except for \Cref{sc:Grassmann products}, which includes complex spaces. Any changes needed for the complex case are indicated in footnotes (only in \Cref{sc:preliminaries}).}, $X$ is a $n$-dimensional Euclidean space. A \emph{$p$-subspace} is a subspace of dimension $p$.
For a multivector $M\in\bigwedge X$, $\inner{M}_p \in \bigwedge^p X$ denotes its component of grade $p$.

A \emph{$p$-blade} ($p=1,2,\ldots$) is a simple multivector $B=v_1\wedge\cdots\wedge v_p \in\bigwedge^p X$, with $v_1,\ldots,v_p\in X$. 
Its \emph{reversion} is $\tilde{B} = v_p\wedge\cdots\wedge v_1 = (-1)^{\frac{p(p-1)}{2}} B$, which extends linearly to an involution of $\bigwedge X$.
If $B\neq 0$, its $p$-subspace is $[B]=\Span(v_1,\ldots,v_p)$, and $\bigwedge^p[B]=\Span(B)$.
% or as the \emph{annihilator} of $B$,
%\[ V=\Ann(B)=\{x\in X:x\wedge B=0\}. \]
A  scalar $B\in\bigwedge^0 X$ is a $0$-blade, with $[B]=\{0\}$. 
Note that $B=0$ is a $p$-blade for all $p$, and $[0]=\{0\}$. % \cite[p.45]{Dorst2007} uses $\emptyset$

For $M,N\in\bigwedge X$, the \emph{scalar product} \cite{Hestenes1984clifford} $M*N = \inner{MN}_0$ is related to the Grassmann algebra inner product by\footnote{In the complex case, the $M$ on the left is also conjugated, but as the same happens wherever we use this product, one can simply always replace $\tilde{M}*N$ with $\inner{M,N}$.}\ $\tilde{M}*N=\inner{M,N}$, and the norm is $\|M\| = \sqrt{\tilde{M}*M}$.
For $B=v_1\wedge\cdots\wedge v_p$, $\|B\| = \sqrt{\tilde{B}B} = \sqrt{\det (v_i\cdot v_j)}$ gives\footnote{For complex blades, $\|B\|^2$ gives the $2p$-dimensional volume of the parallelotope spanned by $v_1,\im v_1,\ldots,v_p, \im v_p$.\label{fn:complex blade norm}}
the $p$-dimensional volume of the parallelotope spanned by $v_1,\ldots,v_p$. 

For subspaces $V,W\subset X$, $\Proj_W:X\rightarrow W$ and $\Proj^V_W:V\rightarrow W$ are orthogonal projections.
We also write $P_B$ for $P_{[B]}$, and extend $P=\Proj_W$ to an orthogonal projection $P:\bigwedge X \rightarrow \bigwedge W$, with $P(v_1\wedge\cdots\wedge v_p) = Pv_1\wedge\cdots\wedge Pv_p$.

A little more geometric algebra notation, for those who are mainly interested in \Cref{sc:Grassmann products}: for vectors $v,w\in X$, the (dot) product $v\cdot w$ equals the Grassmann algebra inner product $\inner{v,w}$, and the Clifford geometric product (indicated by juxtaposed elements, with no product symbol between them) is $vw=v\cdot w+v\wedge w$.
When $v_1,\ldots,v_k\in X$ are orthogonal, the geometric product $v_1 v_2\cdots v_k$  equals the exterior product $v_1\wedge v_2\wedge\cdots\wedge v_k$.

\subsection{Principal Angles and Vectors}

In high dimensions, no single scalar angle can fully describe the relative inclination of subspaces. This requires a list of principal angles \cite{Afriat1957,Galantai2006,Gluck1967,Golub2013,Wedin1983}, also called canonical or Jordan angles. These angles were introduced in 1875 by Jordan \cite{Jordan1875}, and have important applications in statistics (in Hotelling's theory of canonical correlations \cite{Hotelling1936}), numerical analysis and other areas.

\begin{definition}\label{pr:inner ei fj}
	Let $V,W\subset X$ be nonzero subspaces, $p=\dim V$, $q=\dim W$ and $m=\min\{p,q\}$.
	Their \emph{principal angles} are $0\leq \theta_1\leq\cdots\leq\theta_m\leq\frac \pi 2$, and orthonormal bases $\beta_V = (e_1,\ldots,e_p)$ and $\beta_W = (f_1,\ldots,f_q)$ are \emph{associated principal bases}, formed by \emph{principal vectors}, if, for all $1\leq i \leq p$ and $1\leq j \leq q$,
	\begin{equation}\label{eq: ei fj}
		e_i\cdot f_j = \delta_{ij} \cos\theta_i.
	\end{equation}
\end{definition}

We also say that $\beta_W$ is a principal basis of $W$ w.r.t.\! $V$.
Note that
\begin{equation}\label{eq:ProjVW}
	\Proj_W e_i=\begin{cases}
		f_i\,\cos\theta_i \hspace{4pt} \text{ if } i\leq m, \\
		0 \hspace{34pt}\text{ if } i>m.
	\end{cases}
\end{equation}

Principal bases and angles can be obtained via a singular value decomposition, in a method introduced by Björck and Golub \cite{Bjorck1973}  (see also \cite{Drmac2000,Galantai2006,Golub2013}). The $e_i$'s and $f_i$'s are orthonormal eigenvectors of $P^W_V P^V_W$ and $P^V_W P^W_V$, respectively, and the $\cos\theta_i$'s are the square roots of the eigenvalues of $P^W_V P^V_W$, if $p\leq q$, or $P^V_W P^W_V$ otherwise (Fig. \ref{fig:principal angles}).

\begin{figure}
	\centering
	\includegraphics[width=0.8\linewidth]{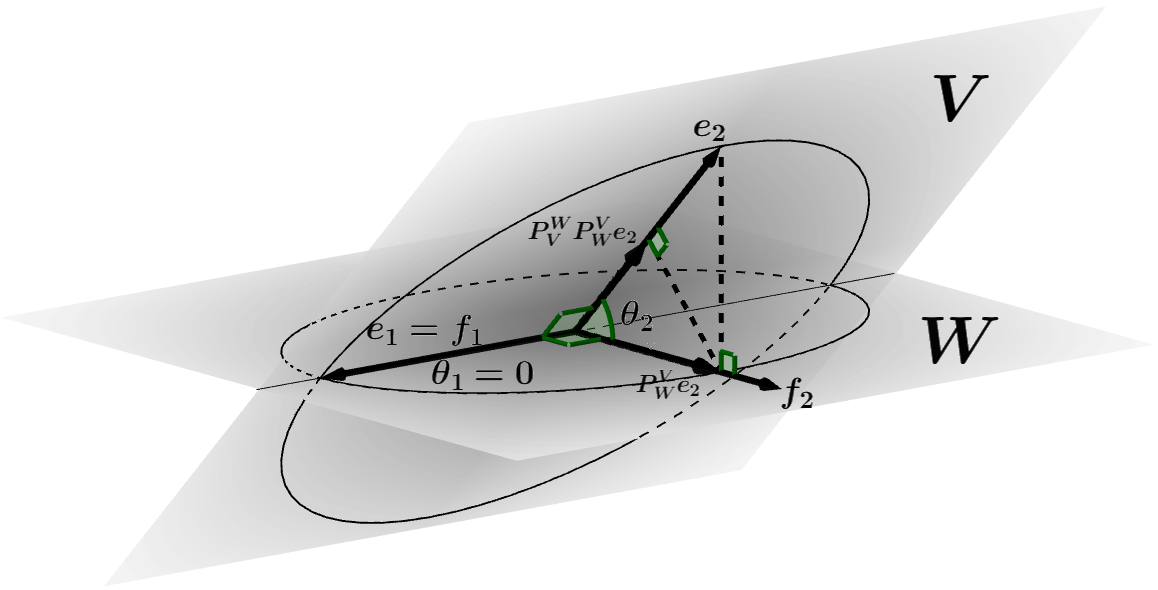}
	\caption{Principal angles and vectors for planes $V,W\subset \R^3$.}
	\label{fig:principal angles}
\end{figure}

A recursive characterization of principal angles is that, for $1\leq i\leq m$,  
$\theta_i=\min\{\theta_{v,w} : v\in V\backslash\{0\}, w\in W\backslash\{0\}, v\cdot e_j=w\cdot f_j=0 \  \forall j<i\}$, where $\theta_{v,w}$ is the angle\footnote{In complex spaces there are different concepts of angle between vectors \cite{Scharnhorst2001}. For the present characterization, Euclidean or Hermitian angles give the same result \cite{Galantai2006}.} 
between $v$ and $w$, and unit vectors $e_i\in V$ and $f_i\in W$ are chosen so that \eqref{eq: ei fj} holds for all $j\leq i$. 
A geometric interpretation is that the unit sphere of $V$ projects orthogonally to an ellipsoid in $W$ and, for $i\leq m$, $e_i$ projects onto a semi-axis of length $\cos\theta_i$ along $f_i$ (Fig. \ref{fig:principal angles}).
%(in the complex case, for each $\theta_i$ there are two semi-axes of equal lengths, along $f_i$ and $\im f_i$).

Note that the number of null principal angles, for which $e_i=f_i$, equals $\dim (V\cap W)$. Also, $V\perp W$ if, and only if, all principal angles are $\frac\pi2$.

Principal angles are uniquely determined, but principal bases are not.
If $\theta_i$ is not repeated then $e_i$ and $f_i$ are determined up to a sign, but if some $\theta_i$'s are equal then orthogonal transformations of the corresponding eigenspaces in $V$ and $W$ yield new principal bases.
For $i\leq m$ with $\theta_i\neq \frac\pi2$, a choice of $e_i$ determines $f_i = P_W e_i/\|P_W e_i\|$.
Any $e_i$'s or $f_i$'s with $i>m$ can be chosen freely to complete an orthonormal basis.

\subsection{Partial Orthogonality}

Let $V,W\subset X$ be subspaces. 
As usual, we write 
%$v\perp W$, for $v\in X$, if $\inner{v,w}=0$ for all $w\in W$, and $W^\perp=\{v\in X:v\perp W\}$. Also, 
$V\perp W$ if $v\cdot w=0$ for all $v\in V$ and $w\in W$, and $W^\perp$ is the orthogonal complement of $W$.
We will also need a weaker concept of orthogonality \cite{Mandolesi_Grassmann}.

\begin{definition}
	$V$ is \emph{partially orthogonal} to $W$ (we write $V\pperp W$) if there is a nonzero $v\in V$ such that $v\cdot w=0$ for all $w\in W$. 
\end{definition}

\begin{proposition}\label{pr:orthogonality}
	For any subspaces $V,W\subset X$:  
	\begin{enumerate}[i)]
		%\item $\{0\}\not\pperp W$.
		%\item $V\pperp\{0\} \ \Leftrightarrow\ V\neq\{0\}$.
		%\item $V \pperp W \ \Leftrightarrow\ V\cap W^\perp\neq\{0\}$.
		\item $V \pperp W \ \Leftrightarrow\ \dim V>\dim W$ or a principal angle of $V$ and $W$ is $\frac \pi 2$.\label{it:pperp dim pi/2}
		\item $V \pperp W \ \Leftrightarrow\ \dim V > \dim P_W(V)$. \label{it:dim P(V)}
		\item If $\dim V=\dim W$ then $V \pperp W  \Leftrightarrow W\pperp V$. \label{it:pperp equal dim}
		%\item $V \pperp W $ and $W\pperp V\ \Leftrightarrow$ a principal angle is $\frac \pi 2$.
		%\item $V\perp W\ \Leftrightarrow\ V=\{0\}$, $W=\{0\}$ or all principal angles are $\frac \pi 2$. 
		%\item $V\perp W\ \Rightarrow\ V \pperp W$ if $V\neq\{0\}$. 
		%\item If $\theta_{V,W}$ is defined then  $V\pperp W \Leftrightarrow \theta_{V,W}=\frac \pi 2$. 
		%\item $V\perp W\ \Rightarrow\ \theta_{V,W}=\frac \pi 2$ if $V\neq\{0\}$. If $\dim V=1$ the converse holds.
		%\item $V'\subset V, \ V'\pperp W \Rightarrow V\pperp W$. 
		%\item $W'\subset W , \ V\pperp W \Rightarrow V\pperp W'$. 
	\end{enumerate}
\end{proposition}
\begin{proof}
	\emph{(\ref{it:pperp dim pi/2})} When $\dim V\leq \dim W$, the largest principal angle equals the largest angle between a nonzero $v\in V$ and $W$.
	\emph{(\ref{it:dim P(V)})} Immediate.
	\emph{(\ref{it:pperp equal dim})} Follows from \emph{\ref{it:pperp dim pi/2}}.
\end{proof}

Note that $\pperp$ is not a symmetric relation when $\dim V\neq\dim W$: any plane is partially orthogonal to a line, but the converse is not true.

%\begin{proposition}\label{pr:P(V) not partial perp}
%If $V\not\pperp W$  and $(f_1,\ldots,f_q)$ is a principal basis of $W$ w.r.t.\! $V$ then $P_W(V)=\Span(f_1,\ldots, f_p)$, where $p=\dim V$.
%\SELF{$p\leq q$ as $V\not\pperp W$}
%%, and the principal angles of  $V$ and $P(V)$ are the same as those of $V$ and $W$.
%\end{proposition}

\begin{definition}
	Let $ A, B\in\bigwedge X$ be blades. 
	\begin{enumerate}[i)]
		%\item If $[ A]\subset [ B]$ we say $ A$ is a \emph{subblade} of $ B$. 
		\item If $[ A]\pperp [ B]$ we say $ A$ is \emph{partially orthogonal} to $ B$, and write $ A\pperp  B$.\label{it:partial orthog}
		\item If $[ A]\perp [ B]$ we say $ A$ and $ B$ are \emph{completely orthogonal}. \label{it:complete orthog}
	\end{enumerate}
\end{definition}

Orthogonality in the sense of $A*B=0$ is weaker than (\ref{it:complete orthog}), and for nonzero blades of same grade it equals (\ref{it:partial orthog})  (see \Cref{pr:pperp A*B=0} below).

%For arbitrary grades, $ A\pperp  B \Leftrightarrow  A\glcontr B=0$, as \eqref{eq:geom lcontr Theta} will show.

\subsection{Principal Decomposition and Relative Orientation}

Let $ A\in\bigwedge^p X$ and $ B\in\bigwedge^q X$ be nonzero blades, and $\beta_A = (e_1,\ldots,e_p)$ and $\beta_B = (f_1,\ldots,f_q)$ be associated principal bases of $[A]$ and $[B]$.

\begin{definition}
	A \emph{principal decomposition} of $ A$ and $ B$ is
	\begin{equation}\label{eq:principal decomposition}
		\begin{aligned}
			A &= \epsilon_A\| A\| e_1 e_2\cdots e_p, \\
			B &=  \epsilon_ B\| B\| f_1 f_2\cdots f_q,
		\end{aligned}
	\end{equation}
	where $\epsilon_ A, \epsilon_ B = \pm 1$. 
	We also define\footnote{In complex spaces $\epsilon_ A$ and $\epsilon_ B$ are phase factors $e^{i\varphi}$, and we define $\epsilon_{ A, B}=\bar{\epsilon}_ A\,\epsilon_ B$.}
	$\epsilon_{ A, B}=\epsilon_ A\epsilon_ B$, which we call \emph{relative orientation} of $A$ and $B$ (w.r.t. $\beta_A$ and $\beta_B$). 
\end{definition}

\begin{lemma}\label{pr:A*B theta_i}
	If $p=q$ then $\tilde{A}*B = \epsilon_{A,B} \|A\|\|B\|\prod_{i=1}^p \cos\theta_i$. 
\end{lemma}
\begin{proof}
	Follows from \eqref{eq: ei fj} and \eqref{eq:principal decomposition}.
\end{proof}

\begin{proposition}\label{pr:pperp A*B=0}
	For nonzero blades of same grade $A, B\in\bigwedge^p X$ we have $ A\pperp  B \Leftrightarrow  A*B=0$. 
\end{proposition}
\begin{proof}
	Follows from \Cref{pr:A*B theta_i} and \Cref{pr:orthogonality}\emph{\ref{it:pperp dim pi/2}}.
\end{proof}

If $A*B\neq 0$
\SELF{so $p=q$}
then $\epsilon_{ A, B}=\frac{\tilde{A}*B}{|A*B|}$ is uniquely determined by the orientations of $A$ and $B$, being such that $\epsilon_{A,B} P_{B} A$ has the same orientation as $B$ (i.e., $\epsilon_{A,B} P_{B} A = cB$ for some $c>0$).
\SELF{$\epsilon_{ A, B}P A=\frac{\inner{A, B}}{|\inner{A, B}|}\inner{ B, A} B/\| B\|^2 = |\inner{ A, B}| B/\| B\|^2 $}
If $A*B= 0$ (distinct grades, or partially orthogonal blades), the orientations of $A$ and $B$ can not really be compared, and $\epsilon_{A, B}$ becomes less meaningful, depending on the choice of principal bases, but in a way that makes it useful to track orientation changes in them.

%If $\tilde{A}*B\geq 0$ then $\epsilon_A = \epsilon_B = 1$ for some principal bases, and we say $A$ and $B$ have \emph{aligned orientations}.
%	\SELF{Inclui o caso $\inner{ A, B}=0$ porque também permite eliminar os $\epsilon$'s}
%Note that this relation is not transitive. 

\subsection{Asymmetric Angle of Subspaces}\label{sc:asymmetric angle}

The asymmetric angle\footnote{Previously \cite{Mandolesi_Grassmann,Mandolesi_Pythagorean} called \emph{Grassmann angle}, for its links with Grassmann algebra, but this has given the false idea that it is a multivector. Since what sets this angle apart from similar ones is its asymmetry, the new name is more adequate.}
\cite{Mandolesi_Grassmann} of subspaces $V,W\subset X$ measures their separation  in terms of a projection factor $\pi_{V,W}$ \cite{Mandolesi_Pythagorean}, describing how volumes in $V$ contract when orthogonally projected on $W$.

\begin{definition}
	Let\footnote{In complex spaces let $p=\dim V_\R=2\dim V$, where $V_\R$ is the underlying real space.} 
	$p=\dim V$, $S\subset V$ be a $p$-dimensional parallelotope, and $\vol_p$ be the $p$-dimensional volume. The \emph{projection factor} of $V$ on $W$ is $\pi_{V,W}=\frac{\vol_p P_W(S)}{\vol_p S}$.
%	\begin{equation*}
%		\pi_{V,W}=\frac{\vol_p P_W(S)}{\vol_p S}.
%	\end{equation*} 
\end{definition}

\begin{definition}
	The \emph{asymmetric angle}	$\Theta_{V,W}\in[0,\frac \pi 2]$ of $V$ with $W$ is given\footnote{In complex spaces we define $\cos^2\Theta_{V,W}=\pi_{V,W}$, to match the relation between blade norm and volume (\cref{fn:complex blade norm}). This preserves properties of the angle, but changes its interpretation, as it is not the angle of the underlying real subspaces \cite{Mandolesi_Grassmann}. Alternative (but less intuitive) definitions for real and complex spaces are \Cref{pr:properties Grassmann}\,\emph{\ref{it:Theta Proj}}, \emph{\ref{it:Theta in Lambda}} or \emph{\ref{it:Theta prod cos}}.} 
	by $\cos\Theta_{V,W}=\pi_{V,W}$.
	%	We also define $\Theta_{\{0\},\{0\}}=0$, $\Theta_{\{0\},V}=0$ and $\Theta_{V,\{0\}}=\frac \pi 2$.
\end{definition}

In simple cases having a clear and unique concept of the angle between the subspaces, $\Theta_{V,W}$ coincides with it (e.g., when $V$ is a line, or $V$ and $W$ are hyperplanes).
This angle has many useful properties \cite{Mandolesi_Grassmann}, of which we mention just a few.

\begin{proposition}\label{pr:properties Grassmann}
	Given subspaces $V,W\subset X$ with principal angles $\theta_1,\ldots,\theta_m$, where $m=\min\{p,q\}$ for $p=\dim V$ and $q=\dim W$, let $A,B\in\bigwedge X$ be nonzero blades such that $[A]=V$ and $[B]=W$. Then:
	\begin{enumerate}[i)]
		\item $\cos\Theta_{V,W}=\frac{\|P_B A\|}{\| A\|}$. \label{it:Theta Proj}
		\item $\Theta_{V,W} = \Theta_{V,P_W(V)}$. \label{it:Theta PWV}
		\item $\Theta_{V,W}$ is the angle\footnote{The usual angle between a line $L$ and a subspace $U$, defined (even in complex spaces \cite{Scharnhorst2001}) as that between a nonzero $v\in L$ and its projection on $U$ (or $\frac \pi 2$ if $P_U v = 0$).}
%		 This also works in complex spaces as all nonzero vectors in a complex line make the same (Euclidean) angle with their projections on a complex subspace \cite{Scharnhorst2001}.} 
		in $\bigwedge^p X$ between $\bigwedge^p V$ and $\bigwedge^p W$. \label{it:Theta in Lambda}
		\item If $\dim V=\dim W$ then $\cos\Theta_{V,W}=\frac{|A*B|}{\|A\|\|B\|}$, and $\Theta_{V,W}=\Theta_{W,V}$. \label{it:Theta star} 
		\item If $p>q$ then $\Theta_{V,W}=\frac\pi 2$, otherwise \label{it:Theta prod cos}
		\begin{equation}\label{eq:Theta prod cos}
			\cos\Theta_{V,W}=\prod_{i=1}^m \cos\theta_i.
		\end{equation}
		\item $\Theta_{V,W}=0 \ \Leftrightarrow\ V\subset W$.\label{it:Theta 0}
		\item $\Theta_{V,W}=\frac \pi 2 \ \Leftrightarrow\ V \pperp W$.\label{it:Theta pi2}
%		\item $\Theta_{\R v,\R w}=\min\{\theta_{v,w},\pi-\theta_{v,w}\}$  in the real case, and $\Theta_{\C v,\C w}=\gamma_{v,w}$ in the complex one.
%		\item $\cos^2\Theta_{V,W}=\det(\bar{\mathbf{P}}^T \mathbf{P})$, where $\mathbf{P}$ is a matrix for $P$ in orthonormal bases of $V$ and $W$.
	\end{enumerate}
\end{proposition}
\begin{proof}
	\emph{(\ref{it:Theta Proj})} Follows from the relation between blade norm and volume.
	\emph{(\ref{it:Theta PWV})} Follows from \emph{\ref{it:Theta Proj}}.
	\emph{(\ref{it:Theta in Lambda})} $\bigwedge^p V = \Span(A)$ and, by \emph{\ref{it:Theta Proj}}, $\Theta_{V,W}$ is the angle between $A$ and its projection on $\bigwedge^p W$.
	\emph{(\ref{it:Theta star})} In this case $\bigwedge^p W = \Span(B)$ and so, by \emph{\ref{it:Theta in Lambda}}, $\Theta_{V,W}$ is the smallest angle between $A$ and $\pm B$.
	\emph{(\ref{it:Theta prod cos})} Follows from \emph{\ref{it:Theta Proj}}, since $P_B A=0$ if $p>q$, otherwise $\|P_B A\| = \|A\| \prod_{i=1}^m \cos\theta_i$, by \eqref{eq:ProjVW} and \eqref{eq:principal decomposition}.
	\emph{(\ref{it:Theta 0})} Follows from \emph{\ref{it:Theta prod cos}}.
	\emph{(\ref{it:Theta pi2})} Follows from \emph{\ref{it:Theta prod cos}} and \Cref{pr:orthogonality}\emph{\ref{it:pperp dim pi/2}}.
\end{proof}

%\begin{proposition}\label{pr:Theta prod cos}
%	Let $V,W\subset X$ have principal angles $\theta_1,\ldots,\theta_m$, where $m=\min(p,q)$ for $p=\dim V$ and $q=\dim W$. If $p>q$ then $\Theta_{V,W}=\frac\pi 2$, otherwise
%	\begin{equation}\label{eq:Theta prod cos}
%		\cos\Theta_{V,W}=\prod_{i=1}^m \cos\theta_i.
%	\end{equation}
%\end{proposition}

As its name indicates, $\Theta_{V,W}$ is asymmetric: in general, $\Theta_{V,W} \neq \Theta_{W,V}$ if $\dim V\neq \dim W$.
This feature sets it apart from similar angles which also measure volume contraction, but in projections from the smaller subspace to the larger one (e.g., those in \cite{Gluck1967,Gunawan2005,Jiang1996,Wedin1983} and in \Cref{sc:GA angles}).
This asymmetry may seem odd, but it reflects the lack of symmetry between subspaces of distinct dimensions.
A way to understand it is to note that, as \Cref{pr:properties Grassmann}\emph{\ref{it:Theta 0}} indicates, $\Theta_{V,W}$ measures, in a sense, how far $V$ is from being contained in $W$. If $\dim V > \dim W$, no rotation of $V$ will bring this any closer to happening, and as $V$ will always have a nonzero vector orthogonal to $W$, $\Theta_{V,W}$ will never be less than $\frac\pi 2$, by \Cref{pr:properties Grassmann}\emph{\ref{it:Theta pi2}}. On the other hand, $W$ can be rotated into $V$, and $\Theta_{W,V}$ can assume any value in $[0,\frac\pi2]$.

The asymmetry turns out to be useful, leading to more general results with simpler proofs, as the angle `keeps track' of special cases depending on which subspace is larger.
For example, if we had defined the angle in a symmetric way (projecting from the smaller subspace to the larger one, as in \cite{Dorst2007}, or using \eqref{eq:Theta prod cos} without the exception for $p>q$, as in \cite{Hitzer2010a}), the results in \Cref{pr:properties Grassmann} would require extra conditions.
%For example, this allows $\Theta_{V,W}$ to satisfy a triangle inequality which holds even for distinct dimensions \cite{Mandolesi_Grassmann}.
%In \Cref{pr:products angles} we relate the asymmetry of $\Theta_{V,W}$ to that of contractions, and its symmetry for equal dimensions to that of the scalar product.

Like other angles between high-dimensional subspaces, $\Theta_{V,W}$ captures some properties of their relative inclination, but misses other information, and it is important to note its peculiarities. Given two pairs of subspaces $(V,W)$ and $(V',W')$, even if all dimensions are the same and $\Theta_{V,W} = \Theta_{V',W'}$ there may be no orthogonal transformation of $X$ matching the pairs (this requires both to have the same list of principal angles \cite{Gluck1967}). And if $\dim V>1$ then $\Theta_{V,W}$ tends to be larger than any (usual) angle between a line of $V$ and $W$.

\subsubsection{Related Angles}

Other angles related to $\Theta_{V,W}$ are useful at times.

\begin{definition}
	The \emph{max- and min-symmetrized angles} are given, respectively, by 
	$\hat{\Theta}_{V,W} = \max\{\Theta_{V,W},\Theta_{W,V}\}$ and $\check{\Theta}_{V,W} = \min\{\Theta_{V,W},\Theta_{W,V}\}$. 
\end{definition}

Symmetrizing with $\min$ corresponds to projecting from the smaller subspace to the larger one, and leads to worse properties: $\check{\Theta}_{V,W}$ does not satisfy a triangle inequality, while $\hat{\Theta}_{V,W}$ gives a metric on the full Grassmannian of all subspaces of $X$ \cite{Mandolesi_Grassmann}.
On the other hand,  we always have $\hat{\Theta}_{V,W}=\frac\pi2$ for different dimensions, which is less helpful. 
The asymmetric $\Theta_{V,W}$ strikes a good balance between nice properties and useful information, giving the Fubini-Study metric on the Grassmannian of subspaces of a given dimension, and an asymmetric metric on the full Grassmannian  \cite{Mandolesi_Grassmann}.

\begin{definition}
	The \emph{complementary angle} is $\Theta_{V,W}^\perp=\Theta_{V,W^\perp}$.
\end{definition}

The term `complementary' refers to the orthogonal complement $W^\perp$, and this should not be confused with the usual complement of an angle.
When $V$ is a line we do have $\Theta_{V,W}^\perp=\frac \pi 2-\Theta_{V,W}$, 
% and $\cos\Theta_{V,W}^\perp=\sin\Theta_{V,W}$, 
but, in general, the relation between these two angles is complicated \cite{Mandolesi_Grassmann}.
Assuming for simplicity that $V$, $W$ and $W^\perp$ have the same dimension $p>1$, this can be understood by noting that, in the Plücker embedding of the  Grassmannian of $p$-subspaces, these angles measure geodesic distances in the ambient space \cite{Mandolesi_Grassmann}, so if $V$ is not in the geodesic from $W$ to $W^\perp$ then $\Theta_{W,V}+\Theta_{V,W^\perp} > \Theta_{W,W^\perp}=\frac\pi2$.

The following result was proven in \cite{Mandolesi_Grassmann} by showing (among other things) that each $\theta_i\neq 0$ gives a principal angle $\frac\pi2-\theta_i$ of $V$ and $W^\perp$, while a $\theta_i=0$ implies $V\pperp W^\perp$.
We now give a simpler proof using \Cref{pr:norm wedge product} (whose proof, in \Cref{sc:Grassmann products}, uses only \Cref{pr:properties Grassmann}\emph{\ref{it:Theta Proj}}, so there is no circularity).

\begin{proposition}\label{pr:complementary product sines}
	Given subspaces $V,W\subset X$ with principal angles $\theta_1,\ldots,\theta_m$, where $m=\min\{\dim V,\dim W\}$, we have
	\begin{equation}\label{eq:Theta perp prod sin}
		\cos\Theta_{V,W}^\perp=\prod_{i=1}^m \sin\theta_i.
	\end{equation}
\end{proposition}
\begin{proof}
	Let $ A=e_1\wedge\cdots\wedge e_p$ and $ B=f_1\wedge\cdots\wedge f_q$ for principal bases $(e_1,\ldots,e_p)$ and $(f_1,\ldots,f_q)$ of $V$ and $W$, and assume, without loss of generality, $p\leq q$. 
%	As principal planes are mutually orthogonal, 
	\Cref{pr:norm wedge product} gives
	%	and $\|e_i\wedge f_i\|=\sin \theta_i$, 
	$\cos\Theta_{V,W}^\perp = \| A\wedge  B\|$, and with \eqref{eq: ei fj} we obtain $\| A\wedge  B\| =\|e_1\wedge f_1\|\cdots\|e_p\wedge f_p\| \|f_{p+1}\|\cdots\|f_q \|  = \sin\theta_1\cdots\sin\theta_p$.
\end{proof}

%In \Cref{sc:Exterior product} we relate $\Theta^\perp_{V,W}$ to the exterior product, obtaining a simple explanation for why it has the symmetry (\!\emph{\ref{it:symmetry complementary}}) that $\Theta_{V,W}$ does not. We also give an easy proof of the following result, used in \cite{Mandolesi_Grassmann} to obtain (\!\emph{\ref{it:symmetry complementary}}), and whose demonstration was somewhat laborious. 

\begin{proposition}\label{pr:complementary simple cases}
	Let $V,W\subset X$ be subspaces.
	\begin{enumerate}[i)]
		%		\item $\Theta^\perp_{V,\{0\}} = \Theta^\perp_{\{0\},V} =  0$.\label{it:complementary 0}
		%		\item If $V\neq\{0\}$ then $\Theta^\perp_{V,X} = \Theta^\perp_{X,V} = \frac \pi 2$.
		\item $\Theta_{V,W}^\perp=0 \ \Leftrightarrow\ V\perp W$.\label{it:Theta perp 0}
		\item $\Theta_{V,W}^\perp=\frac \pi 2 \ \Leftrightarrow\ V\cap W\neq\{0\}$. \label{it:Theta perp pi 2}
		\item $\Theta_{V,W}^\perp = \Theta_{W,V}^\perp$. \label{it:symmetry complementary}
		%		\item $\Theta_{V,W} = \Theta_{W^\perp,V^\perp}$. \label{it:Theta perp perp}
		%		\item $0\leq \cos^2\Theta_{V,W} + \cos^2\Theta_{V,W}^\perp \leq 1$, if $V\neq\{0\}$. \label{it:sum cos2 Thetas}
%		\item $\Theta_{V,W}^\perp = \Theta_{V,P_W V}^\perp$. \label{it:Theta PWV perp}
	\end{enumerate}
\end{proposition}
\begin{proof}
	\emph{(\ref{it:Theta perp 0})} By \eqref{eq:Theta perp prod sin}, $\Theta_{V,W}^\perp=0  \Leftrightarrow \theta_i=\frac\pi2$ for all $i$.
	\emph{(\ref{it:Theta perp pi 2})} By \eqref{eq:Theta perp prod sin}, $\Theta_{V,W}^\perp=\frac\pi2 \Leftrightarrow \theta_i=0$ for some $i$.
	\emph{(\ref{it:symmetry complementary})} Follows from \eqref{eq:Theta perp prod sin}, as principal angles do not depend on the order of $V$ and $W$. We can also obtain it directly from \Cref{pr:norm wedge product}.
%	\emph{(\ref{it:Theta PWV perp})} If $V\perp W$ both angles are $0$, otherwise the principal angles of $V$ and $P_W V$ are those of $V$ and $W$ which are not $\frac\pi2$, and the result follows from \eqref{eq:Theta perp prod sin}.
\end{proof}

Note that \eqref{eq:Theta perp prod sin} holds regardless of the dimensions of $V$ and $W$, unlike \Cref{pr:properties Grassmann}\,\emph{\ref{it:Theta prod cos}}, and $\Theta^\perp_{V,W}$ is always symmetric.
A way to understand this is to note that $\dim V>\dim W^\perp \Leftrightarrow \dim W>\dim V^\perp$.
What is surprising is that \Cref{pr:complementary simple cases}\,\emph{\ref{it:Theta perp 0}}, \emph{\ref{it:Theta perp pi 2}} and \eqref{eq:Theta perp prod sin} depend on the asymmetry of $\Theta_{V,W}$, without which these results  would not even hold for two planes in $\R^3$.

\subsubsection{Oriented Angles}

It is also convenient to define an angle that takes the orientation of subspaces into account.
Let $V$ and $W$ be oriented by blades $A,B\in\bigwedge X$, respectively, with relative orientation $\epsilon_{A,B}$ (w.r.t. given principal bases of $V$ and $W$).

\begin{definition}\label{df:oriented angle}
	The \emph{oriented asymmetric angle}\footnote{In complex spaces, it is a complex-valued angle $\Theta_{A,B}\in\C$. Complex-valued angles between complex vectors have been considered, for example, in \cite{Scharnhorst2001}.}
	$\Theta_{A,B}\in[0,\pi]$ of $V$ with $W$ is given by $\cos\Theta_{A,B} = \epsilon_{A,B} \cos\Theta_{V,W}$.
\end{definition}

We have $\Theta_{A,B} = \Theta_{V,W}$ if $\tilde{A}*B > 0$, $\Theta_{A,B} = \pi - \Theta_{V,W}$ if $\tilde{A}*B < 0$, and if $\tilde{A}*B = 0$ it depends on the principal bases (which will be useful).

We call $\Theta_{A,B}$ an `oriented angle' for short, as the `oriented' refers to the subspaces.
To be clear about our notation:
any $\Theta_{V,W}$, with subspaces as subscripts, is a non-oriented angle, while any $\Theta_{A,B}$, with blades as subscripts, is the oriented angle of the subspaces $[A]$ and $[B]$, oriented by these blades.
Such $\Theta_{A,B}$ should not be confused with the usual angle between $A$ and $B$ in $\bigwedge X$ (though they do coincide when grades are equal).
For the non-oriented angle of $[A]$ and $[B]$ we write $\Theta_{[A],[B]}$.
The same convention will apply to non-oriented and oriented versions of other concepts.

\begin{definition}\label{df:oriented symm complem}
	$\pi_{A,B} = \epsilon_{A,B} \pi_{V,W}$ is an \emph{oriented projection factor}, and \emph{oriented max-symmetrized} and \emph{complementary angles} $\hat{\Theta}_{A,B}, \Theta_{A,B}^\perp \in[0,\pi]$ are given by $\cos\hat{\Theta}_{A,B} = \epsilon_{A,B} \cos\hat{\Theta}_{V,W}$ and $\cos\Theta_{A,B}^\perp = \epsilon_{A,B} \cos\Theta_{V,W}^\perp$.
\end{definition}

Note that in general $\hat{\Theta}_{A,B} \neq \max\{\Theta_{A,B},\Theta_{B,A}\}$, and for distinct grades $\hat{\Theta}_{A,B} = \frac\pi 2$. Also, $\Theta_{A,B}^\perp$ encodes information about the relative orientation of $V$ and $W$, not $V$ and $W^\perp$ (this will be important in \Cref{pr:detailed subproducts}).

\subsubsection{Geometric Algebra Angles}\label{sc:GA angles}

The geometric algebra literature has different definitions for the angle between blades or subspaces, all closely related to our angles.

Hestenes \cite{Hestenes1984clifford}
\SELF{p. 14}
defines an angle between multivectors by $\cos\phi=\frac{\tilde{A}*B}{\|A\|\|B\|}$. He says it has a simple geometric interpretation  for same grade blades, but only describes it for intersecting planes, as a dihedral angle. 
For same grade blades it equals $\Theta_{A,B}$, and for distinct grades it is always $\frac\pi 2$ (even for a line contained in a plane), so that it corresponds to $\hat{\Theta}_{A,B}$.

Dorst et al. \cite{Dorst2007} use Hestenes definition for same grade blades, and if $A$ has a lower grade than $B$ they take the angle with its projection on $B$, which is the same as $\Theta_{[A],[B]}$.
They erroneously see it as a dihedral angle: if, after taking out common factors, there is at most one vector left in each blade, it is the angle between these vectors, otherwise ``no single scalar angle can be defined geometrically, and this geometric nonexistence is reflected in the algebraic answer of $0$ for the scalar product'' \cite[p. 70]{Dorst2007}. This is incorrect: if $(e_1,e_2,e_3,e_4)$ is the canonical basis of $\R^4$, $A=(e_1+e_2)\wedge(e_3+e_4)$ and $B=e_1\wedge e_3$ have no common factors (as $A\wedge B\neq 0$) but $A*B\neq 0$.

Hitzer \cite{Hitzer2010a} defines the angle for subspaces of same dimension as in \eqref{eq:Theta prod cos}, but uses it for different dimensions as well, without the exception for $p>q$, so it corresponds to the min-symmetrized angle $\check{\Theta}_{V,W}$.
% though at a point he excludes any $\theta_i=\frac \pi 2$. 
He is silent on its geometric interpretation, and recovers Hestenes formula for equal dimensions.

In a survey of the theory \cite{Macdonald2017}, the angle is defined in terms of a ratio of volumes, by $\cos\phi =\frac{\|P_B A\|}{\|A\|}$, which would make it equal to $\Theta_{[A],[B]}$. But it is not clear if the case of $A$ having larger grade is admitted.

\section{Angle Bivector of Subspaces}\label{sc:Angle bivector}

As seen, the asymmetric angle has many useful properties, but it does not fully describe the relative inclination of subspaces (no single scalar angle can do this).
Alternatively, an angle bivector can conveniently store all data about the relative inclination of two subspaces of same dimension. In \Cref{sc:Subspaces of different dimensions} we consider the case of distinct dimensions.

Let $V,W\subset X$ be $p$-subspaces, having associated principal bases $\beta_V=(e_1,\ldots,e_p)$ and $\beta_W=(f_1,\ldots,f_p)$, and  principal angles $\theta_1\leq \cdots\leq\theta_p$.
Also, let $d = \dim (V\cap W)$, $E=e_1e_2\cdots e_p$ and $F=f_1f_2\cdots f_p$.

\begin{definition}
	For $1\leq i\leq p$, $\Span(e_i,f_i)$ is a \emph{principal plane}, with \emph{principal bivector} $\I_i$ (oriented from $V$ to $W$) given by
	\SELF{O caso $i\leq d$ ajuda na fórmula de Hitzer, e em \cref{ex:geom product 2}}
	\SELF{$R_i = e_i f_i = e^{\I_i \theta_i}$}
	\begin{equation*}
		\I_i =\begin{cases}
			\hspace{3pt} 0 \hspace{35pt} \text{if } i\leq d, \\
			\frac{e_i \wedge f_i}{\|e_i \wedge f_i\|} \quad \text{ if } i>d.
		\end{cases}
	\end{equation*}
	%	and a \emph{principal rotor} $R_i = e_i f_i = e^{\I_i \theta_i}$.
	For $d< i\leq p$, $e_i^\perp = \I_i f_i$ and $f_i^\perp = e_i\I_i$ are \emph{orthoprincipal vectors}.
	\SELF{$I=e^\perp f=ef^\perp$}
\end{definition}

\begin{figure}
	\centering
	\includegraphics[width=0.7\linewidth]{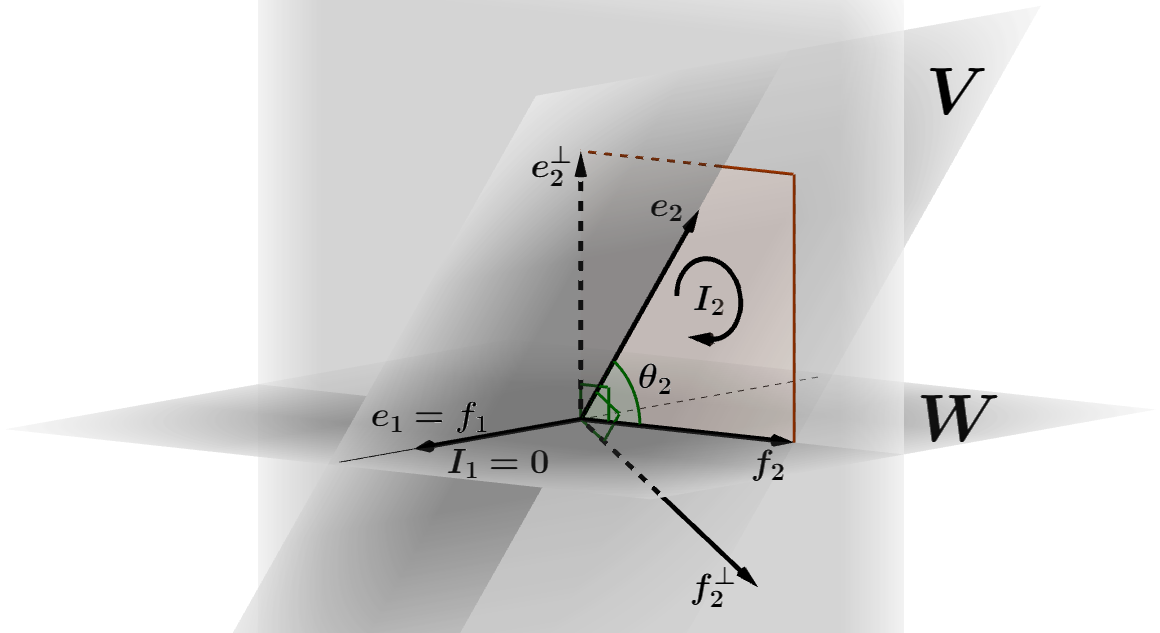}
	\caption{Principal vectors, bivectors and planes, and orthoprincipal vectors, for planes $V,W\subset \R^3$.}
	\label{fig:principals}
\end{figure}

The term `plane' is used broadly, as it degenerates to a line if $i\leq d$.
%, and in the complex case $\dim_\R \mathcal{P}_i = 4$ if $i>d$. % it can have real dimension $4$.
For $d< i\leq p$, $e_i^\perp$ is the normalized component of $e_i$ orthogonal to $W$, and likewise for $f_i^\perp$ (see Fig. \ref{fig:principals}).
Note that $\I_i = e_i^\perp f_i = e_i f_i^\perp$ anti-commutes with $e_i$ and $f_i$, and for $j\neq i$ it commutes with $e_j$, $f_j$ and $\I_j$, since principal planes are mutually orthogonal.

\begin{definition}\label{df:angle bivector}
	$\btheta_{V,W}= \sum_{i=1}^p \theta_i\I_i$ is an \emph{angle bivector} from $V$ to $W$.
\end{definition}

Note that $\tilde{\btheta}_{V,W} = - \btheta_{V,W}$ is an angle bivector from $W$ to $V$.
Also, $\btheta_{V,W}$ may depend on the choice of principal bases.
In \Cref{sc:Geodesics in Grassmannians} we interpret this non-uniqueness in terms of the geometry of the Grassmannians.

\begin{proposition}\label{pr:uniqueness thetaVW}
	$\btheta_{V,W}$ is uniquely defined $\,\Leftrightarrow\, \theta_p\neq\frac\pi2$.
\end{proposition}
\begin{proof}
	Switching signs of both $e_i$ and $f_i$ does not affect $I_i$.
	If $\theta_p = \frac\pi2$ we can keep $e_p$ and replace $f_p$ with $-f_p$, so now $I_p$ switches sign.
	If $\theta_p\neq\frac\pi2$ but some $\theta_i$'s are repeated, the choice of principal bases can affect their $I_i$'s, but does not change $\btheta_{V,W}$.
	To prove it, let $\theta_i=\theta\neq \frac\pi 2$ for all $i$, for simplicity. 
	An orthogonal change of principal basis $e_i' = \sum_j c_{ij}e_j$ in $V$
	\SELF{$(c_{ij})$ orthogonal, so $\sum_i c_{ij}c_{ik} = \delta_{jk}$}
	must 
	\SELF{$\theta\neq \frac\pi 2$ forces it, as $f_i' = P_w e_i'/\|P_w e_i'\|$}
	be accompanied by a corresponding change $f_i' = \sum_j c_{ij}f_j$ in $W$, so the bases remain associated ($e_i'\cdot f_j' = \delta_{ij} \cos\theta$). A calculation shows $\sum_i \theta I_i' = \sum_i \theta I_i$ for the new $I_i' = \sum_{j,k} c_{ij} c_{ik} \frac{e_j \wedge f_k}{\sin \theta}$. 
\end{proof}

With more than one $\theta_i=\frac\pi2$, we can obtain a continuous family of $\btheta_{V,W}$'s via independent orthogonal transformations of $V\cap W^\perp$ and $W\cap V^\perp$.

In \Cref{sc:Exponentials of Angle Bivectors} we analyze the exponentials of angle bivectors in detail. For now, we show that they give rotors connecting the subspaces.

\begin{proposition}\label{pr:theta rotates V to W}
	$F = e^{-\frac \btheta 2} E e^{\frac \btheta 2} = E e^{\btheta} = e^{-\btheta} E$, where $\btheta = \btheta_{V,W}$.
	\SELF{Convenção de Hestenes1984clifford,Hitzer2010a. Outros usam $R=e^{-I\theta}$, que atua como $RxR^{-1}$, mas causa mais sinais}
\end{proposition}
\begin{proof}
	Commutativity and $f_i = e^{-\I_i \theta_i/2} \,e_i\, e^{\I_i \theta_i/2}$ give the first equality.
	\SELF{$e^{\frac 12 \btheta_{V,W}} = e^{\I_1 \theta_1/2}\cdots e^{\I_p \theta_p/2}$}
	And since $e_i$ anti-commutes with $\I_i$ we have $e^{-\I_i \theta_i/2} \,e_i = e_i\,e^{\I_i \theta_i/2}$.
\end{proof}

The operation $E \mapsto e^{-\frac 12 \btheta_{V,W}} E e^{\frac 12 \btheta_{V,W}}$ takes $V = [E]$ onto $W = [F]$ via independent rotations, by principal angles along principal planes, each produced by a plane rotor $e^{\I_i \theta_i/2}$.
The simpler process $E\mapsto E e^{\btheta_{V,W}}$ also takes $V$ to $W$, but while  $e^{-\frac12\btheta_{V,W}} e_i e^{\frac12\btheta_{V,W}} = f_i$, in general $e_i e^{\btheta_{V,W}} \neq f_i$.
\SELF{$e_i e^{\btheta_{V,W}} = e_i e^{\theta_i I_i} e^{\ldots} = f_i e^{\ldots}$} 
So, while the first operation takes each vector of $V$ to another in $W$, the second one relates the whole subspaces, without mapping individual vectors.

Since $E$ and $F$ are unit $p$-blades in $V$ and $W$, \Cref{pr:theta rotates V to W} implies $e^{\btheta_{V,W}}$ (but not $e^{\frac12\btheta_{V,W}}$) is uniquely defined up to a sign,
	\SELF{due to the distinct orientations of $W$ to which $V$ can be taken}
even when $\btheta_{V,W}$ is not unique.

\begin{example}\label{ex:angle bivector}
	Let $\{e_1,e_2,e_3,f_2,f_3\}$ be an orthonormal set, $E=e_1e_2e_3$, $F=e_1f_2f_3$, $V=[E]$ and $W=[F]$. 
	With the associated principal bases $\beta_V=(e_1,e_2,e_3)$ and $\beta_W=(e_1,f_2,f_3)$ we obtain $\btheta_{V,W} = \frac\pi2 (e_2f_2 + e_3f_3)$, 
	$e^{\btheta_{V,W}} = e_2f_2e_3f_3$ and
	$e^{\frac12\btheta_{V,W}} = (1+e_2f_2+e_3f_3+e_2f_2e_3f_3)/2$.
	Calculations confirm \Cref{pr:theta rotates V to W} and show that $e^{-\frac12\btheta_{V,W}} e_2 e^{\frac12\btheta_{V,W}} = f_2$ but $e_2 e^{\btheta_{V,W}} = f_2 e_3 f_3 \neq f_2$.
	Note that the angle bivector ignores $e_1\in V\cap W$.
	
	If $f_2'=(f_2+f_3)/\sqrt{2}$ and $f_3'=(f_2-f_3)/\sqrt{2}$ then $\beta_W' = (e_1,f_2',f_3')$ is another principal basis of $W$ associated with $\beta_V$.  
	With it, we now find $\btheta_{V,W}' = \frac\pi2 (e_2f_2' + e_3f_3') \neq \btheta_{V,W}$,
	$e^{\btheta_{V,W}'} = e_2f_2'e_3f_3' = -e^{\btheta_{V,W}}$ and
	$e^{\frac12\btheta_{V,W}'} = (1+e_2f_2'+e_3f_3'+e_2f_2'e_3f_3')/2 \neq \pm e^{\frac12\btheta_{V,W}}$.
	Also, $e^{-\frac12\btheta_{V,W}'} e_2 e^{\frac12\btheta_{V,W}'} = f_2'$ and likewise for $e_3$, so $E$ is rotated to $F'=e_1f_2'f_3'=-F$.
\end{example}

\subsection{Oriented Angle Bivector}

Let $V$ and $W$ be oriented by nonzero blades $A,B\in\bigwedge^p X$, respectively, with relative orientation $\epsilon_{A,B}$ (w.r.t. $\beta_V$ and $\beta_W$).

\begin{definition}\label{df:oriented angle bivector}
	The \emph{oriented principal angles and bivectors} of $V$ and $W$ are $\theta_i^+=\theta_i$ and $I_i^+=I_i$ for $i<p$ and also for $i=p$ if $\epsilon_{A,B}=1$, otherwise  $\theta_p^+ = \pi-\theta_p$ and $\I_p^+ = -\I_p$.
	The \emph{oriented angle bivector} from $V$ to $W$ is 
		\SELF{equivalent to $\btheta_{A,B} = \btheta_{V,W} + (\epsilon_{A,B}-1)\frac\pi 2 I_p$, or subtracting $\pi$ in $\theta_p$}
	\begin{equation}\label{eq:oriented angle bivector}
		\btheta_{A,B} = \sum_{i=1}^p \theta_i^+\I_i^+ = \begin{cases}
			\btheta_{V,W} \hspace{31pt} \text{ if } \epsilon_{A,B}=1, \\
			\btheta_{V,W} -\pi I_p \ \text{ if } \epsilon_{A,B}=-1.
		\end{cases}
	\end{equation}	
\end{definition}

Again, we use $\btheta_{V,W}$ for the non-oriented angle, $\btheta_{A,B}$ for the oriented one.
Note that $0\leq \theta_1^+\leq\cdots\leq\theta_{p-1}^+ \leq \frac\pi2$ and $\theta_{p-1}^+ \leq \theta_p^+ \leq \pi - \theta_{p-1}^+$.

\begin{proposition}\label{pr:uniqueness thetaAB}
	$\btheta_{A,B}$ is uniquely defined $\,\Leftrightarrow\, \theta_{p-1}^+ + \theta_p^+ \neq \pi$.
\end{proposition}
\begin{proof}
	$\theta_{p-1}^+ + \theta_p^+ = \pi \Leftrightarrow \theta_{p-1} = \theta_p = \frac\pi 2$, or $\theta_{p-1} = \theta_p \neq \frac\pi2$ and $\epsilon_{A,B} = -1$.
	Swapping $f_{p-1}$ and $f_p$ in the first case, 
	\SELF{$e^\perp_{p-1}f_{p-1} \pm e^\perp_{p}f_{p}$ becomes $e^\perp_{p-1}f_{p} \mp e^\perp_{p}f_{p-1}$ ($\epsilon_{A,B}$ changes)}
	or $(e_{p-1},f_{p-1})$ and $(e_p,f_p)$ in the second one, changes $\btheta_{A,B}$. 
	\SELF{$\btheta_{V,W}$ does not change, and $\btheta_{A,B} = \btheta_{V,W} -\pi I_p$ turns into $\btheta_{V,W} -\pi I_{p-1}$}
	
	If $\theta_{p-1} \neq \theta_p = \frac\pi 2$, replacing $f_p$ with $-f_p$ as in \Cref{pr:uniqueness thetaVW} does not change $I_p^+ = \epsilon_{A,B} I_p$, as both $I_p$ and $\epsilon_{A,B}$ switch signs.
	If $\theta_p \neq \frac\pi 2$ then $\btheta_{V,W}$ is uniquely defined, and so is $I_p$ if $\theta_{p-1} \neq \theta_p$ as well. 
\end{proof}

\begin{proposition}\label{pr:exp theta+ exp theta}
	$e^{\btheta_{A,B}} = \epsilon_{A,B} e^{\btheta_{V,W}}$.
\end{proposition}
\begin{proof}
	Follows from \eqref{eq:oriented angle bivector}.
\end{proof}

\begin{proposition}\label{pr:oriented theta rotates A to B}
	If $\|A\|=\|B\|=1$ then $B = e^{-\frac\btheta 2} A e^{\frac \btheta 2} = A e^{\btheta} = e^{-\btheta} A$, where $\btheta = \btheta_{A,B}$.
\end{proposition}
\begin{proof}
	By \eqref{eq:principal decomposition}, $A = \epsilon_A E$ and $B = \epsilon_B F$, so the result follows from \Cref{pr:theta rotates V to W,pr:exp theta+ exp theta}, and also $e^{\frac12\btheta_{A,B}} = -I_p e^{\frac12\btheta_{V,W}}$ if $\epsilon_{A,B} = -1$.
		\SELF{In this case, $e^{-\frac{\btheta}{2}} A e^{\frac{\btheta}{2}} = -\epsilon_A I_p e^{-\frac{1}{2}\btheta_{V,W}} E I_p e^{\frac{1}{2}\btheta_{V,W}} = \epsilon_B I_p F I_p = -\epsilon_B F I_p^2 = B$}
\end{proof}

Now $V$ is rotated onto $W$ matching orientations, and $e^{\btheta_{A,B}}$ (but not $e^{\frac12\btheta_{A,B}}$) is uniquely determined by $A$ and $B$.
\SELF{This can be understood noting that an $e_i$ with $\theta_i=\frac\pi2$ can in principle be rotated to any unit vector $f_i \in W\cap V^\perp$, but the rotation of $e_p$ is determined by the previous choices and the orientation to be matched.}
In \Cref{sc:Subspaces of different dimensions} we show this is not valid in case of different dimensions.

\begin{example}\label{ex:oriented angle bivector}
	In \Cref{ex:angle bivector}, let $V$ and $W$ be oriented by $E$ and $F$. 
	Using $\beta_V$ and $\beta_W$ we find $\epsilon_{E,F} =1$ and $\btheta_{E,F}=\btheta_{V,W}$.
	As $F'=-F$, with $\beta_W'$ we have $\epsilon_{E,F}'=-1$ and $\btheta_{E,F}' = \frac\pi2 (e_2f_2'-e_3f_3')$, which does not equal $\btheta_{V,W}$ nor $\btheta_{V,W}'$. 
	Still, $e^{\btheta_{E,F}'} = -e_2f_2'e_3f_3' = e^{\btheta_{E,F}}$, so that, either way, $V$ is taken to $W$ with the correct orientation.
	However, $e^{\frac12 \btheta_{E,F}'} \neq e^{\frac12 \btheta_{E,F}}$, since vector-wise we have distinct rotations: $(e_1,e_2,e_3)$ is taken by the first rotor to $(e_1,f_2',-f_3')$, and by the second one to $(e_1,f_2,f_3)$.
\end{example}

\subsection{Minimal Geodesics in the Grassmannians}\label{sc:Geodesics in Grassmannians}

Let $G_p(X)$ (resp. $G_p^+(X)$) be the Grassmannian of non-oriented (resp. oriented) $p$-subspaces of $X$, identified with its Plücker embedding in the projective space $\PP(\bigwedge^p X)$ (resp. unit sphere $S(\bigwedge^p X)$).
%In \cite{Mandolesi_Grassmann} we have shown that $\Theta_{V,W}$ gives the Fubini-Study distance between $V,W\in G_p(X)$, and $\Theta_{A,B}$ does the same for $A,B\in G_p^+(X)$.
We relate $\btheta_{V,W}$ and $\btheta_{A,B}$ to the geometry of these manifolds, using results from \cite{Kozlov2000}.

The curve given by $F(t) = e^{-\frac{t}{2}\btheta_{V,W}} E e^{\frac{t}{2}\btheta_{V,W}} = f_1(t)\wedge\cdots\wedge f_p(t)$, where $t\in[0,1]$ and $f_i(t) = \cos(t\theta_i) e_i + \sin(t\theta_i) f_i^\perp$, is a minimal geodesic in $G_p(X)$ connecting $V$ to $W$, and $\|\btheta_{V,W}\| = \left(\sum_{i=1}^p \theta_i^2\right)^{\frac12}$ gives the arc-length distance between them.
Note that this is the distance along geodesics inside $G_p(X)$, while the Fubini-Study distance given by $\Theta_{V,W}$ measures geodesics in the ambient space $\PP(\bigwedge^p X)$.

A minimal geodesic in $G_p^+(X)$ connecting unit $p$-blades $A$ and $B$ is given by $B(t) = e^{-\frac{t}{2}\btheta_{A,B}} A e^{\frac{t}{2}\btheta_{A,B}} =  f_1(t)\wedge\cdots\wedge f_p(t)$, with $t\in[0,1]$, $f_i(t)$ as before for $i<p$ and $f_p(t) = \cos(t\theta_p) e_p + \epsilon _{A,B} \sin(t\theta_p) f_p^\perp$.
Its length $\|\btheta_{A,B}\|$ equals $\|\btheta_{V,W}\|$ if $\epsilon_{A,B}=1$, otherwise $\|\btheta_{A,B}\|^2 = \|\btheta_{V,W}\|^2 + \pi(\pi-2\theta_p)$.

The minimal geodesic is unique unless $\theta_p=\frac\pi2$ in the non-oriented case,
\SELF{i.e. if $\theta_p<\frac\pi2$ the minimal geodesic is unique}
or $\theta_{p-1}^+ + \theta_p^+ = \frac\pi2$ in the oriented one, precisely the cases in which $\btheta_{V,W}$ or $\btheta_{A,B}$ depend on the choice of principal bases. 
In fact, we have:

\begin{proposition}
	There is a one-to-one correspondence between angle bivectors $\btheta_{V,W}$ (resp. $\btheta_{A,B}$) and minimal geodesics connecting the subspaces in $G_p(X)$ (resp. $G_p^+(X)$).
\end{proposition} 
\begin{proof}
	Any $\btheta_{V,W}$ determines a minimal geodesic in $G_p(X)$ given by $F(t)$ as above. 
	An angle bivector of $V$ with its middle point $U = F(\frac12)$ is given by $\btheta_{V,U} = \btheta_{V,W}/2$, and the principal angles of $V$ and $U$ are at most $\frac\pi 4$. By \Cref{pr:uniqueness thetaVW}, $\btheta_{V,U}$ is uniquely defined, so if $\btheta_{V,W}'$ gives the same geodesic then $\btheta_{V,W}' = 2\btheta_{V,U} = \btheta_{V,W}$.
	
	And given a minimal geodesic $\gamma$ from $V$ to $W$,  the minimal geodesic from $V$ to its middle point $U$ is unique.
	So, given an angle bivector $\btheta_{V,U}$, the geodesic it determines must coincide with the first half of $\gamma$. 
	Thus $\btheta_{V,W} = 2\btheta_{V,U}$ is an angle bivector determining $\gamma$.
	
	The proof for the oriented case is similar, using \Cref{pr:uniqueness thetaAB}.
\end{proof}

\subsection{Exponentials of Angle Bivectors}\label{sc:Exponentials of Angle Bivectors}

Exponentials of angle bivectors decompose into rotors of principal planes, or in terms of principal angles, asymmetric angles, or projection factors, as follows.
In \Cref{sc:Principal angles via geometric algebra} we show how to obtain these decompositions explicitly.

\begin{definition}
	For $1\leq i\leq p$, $R_i = e_i f_i = e^{I_i\theta_i} = \cos\theta_i + I_i \sin\theta_i$ is a \emph{principal rotor}.
\end{definition}

\begin{proposition}\label{pr:exp product rotors}
	$e^{\btheta_{V,W}} = R_1 R_2\cdots R_p$.
\end{proposition}
\begin{proof}
	As the $I_i$'s commute, $e^{\btheta_{V,W}} = \prod_{i=1}^p e^{I_i\theta_i}$.
\end{proof}

%This formula can be expanded as follows.

\begin{proposition}\label{pr:exp expansion theta_i}
	Let $c_i=\cos\theta_i$ and $s_i=\sin\theta_i$. Then
	\begin{align*} %\label{eq:exp expansion}
		e^{\btheta_{V,W}} &= c_1c_2\cdots c_p \\
		&\hspace{4pt} + s_1c_2\cdots c_p \I_1 + c_1s_2c_3\cdots c_p \I_2 + \cdots +c_1\cdots c_{p-1} s_p \I_p \nonumber\\
		&\hspace{4pt} +s_1s_2c_3\cdots c_p \I_1\I_2 + s_1c_2s_3\cdots c_p \I_1\I_3 + \cdots + c_1\cdots s_{p-1}s_p \I_{p-1}\I_p  \nonumber\\
		&\hspace{8pt} \vdots \nonumber\\
		&\hspace{4pt} + c_1s_2\cdots s_p \I_2\cdots I_p + \cdots + s_1\cdots s_{p-1} c_p \I_1\cdots I_{p-1} \nonumber\\
		&\hspace{4pt} + s_1s_2\cdots s_p \I_1\I_2\cdots\I_p. \nonumber
	\end{align*}
\end{proposition}
\begin{proof}
	$e^{\btheta_{V,W}} = \prod_{i=1}^p R_i = \prod_{i=1}^p (c_i+s_i\I_i)$.
\end{proof}

This expression appears in Hitzer's geometric product formula \cite{Hitzer2010a}, and this is understandable since \Cref{pr:theta rotates V to W} gives $e^{\btheta_{V,W}} = \tilde{E}F$. 
We can simplify it using the asymmetric angles and some multi-index notation.

\begin{definition}
	Given $a,b,k\in\N$ with $a\leq b$ and $1\leq k\leq b-a+1$, let $\II_0^{a,b} =\{0\}$ and 
	$\II_k^{a,b} = \{ (i_1,\ldots,i_k)\in\N^k : a\leq i_1 < \cdots<i_k\leq b \}$.
	Also, let $\II^{a,b}= \cup_{k=0}^{b-a+1}\, \II_k^{a,b}$. 
	When $a=1$ we omit it and write $\II_k^b$ and $\II^b$.
	
	%	For $\ii\in\II_0^b$ let $\|\ii\|=0$ and $\ii'=(1,\ldots,b)\in\II_b^b$. 
	%	For $\ii=(i_1,\ldots,i_k)\in\II_k^{b}$ let $\|\ii\|=i_1+\ldots+i_k$ and $\ii'=(1,\ldots,i_1',\ldots,i_k',\ldots,b)\in\II_{b-k}^{b}$, where  each $i_j'$ indicates that index is removed.
	%, and for $\ii\in\bla_b^b$ let $\hat{\ii}=0$. %\in\bla_0^b$.
\end{definition}

%\begin{definition}
%Let $\mathbf{i}_0=1$ and $\mathbf{i}_I = \mathbf{i}_{i_1}\ldots\mathbf{i}_{i_k}$ for $I=(i_1,\ldots,i_k)\in\bla_k^{d+1,p}$. For $I\in\bla^{d+1,p}$, let\footnote{This $W^I$ is not to be confused with the $W_I$ of \Cref{sc:Coordinate decomposition}.} $W^I$ be the outer product null space of $B^I = f_1\ldots f_p\,\tilde{\mathbf{i}}_I$.
%%, otherwise let $W^I=W$.
%%\SELF{for completeness}
%\end{definition}

\begin{definition}
	Let $\I_0=1$ and $\I_\ii = \I_{i_1}\cdots\I_{i_k}$ for $\ii=(i_1,\ldots,i_k)\in\II^{d+1,p}$, where $d=\dim (V\cap W)$. 
	\SELF{$1\leq k\leq p-d$}
	For any $\ii\in\II^{d+1,p}$, let $F_\ii = \I_\ii F$ and $W_\ii = [F_\ii]$, where $F = f_1f_2\cdots f_p$ as before. 
\end{definition}

\begin{lemma}\label{pr:Ii Wi}
	For any $\ii\in\II^{d+1,p}$:
	\begin{enumerate}[i)]
		\item $F_\ii$ is the unit $p$-blade obtained from $F$ by replacing  $f_i$ with $e_i^\perp$ for each $i\in\ii$.\label{it:Fi}
		\item $\beta_\ii = \{f_i:i\not\in\ii\} \cup \{e_i^\perp:i\in\ii\}$ is a principal basis of $W_\ii$ associated to $\beta_V$. \label{it:beta i}
		\item The (unordered) principal angles of $V$ and $W_\ii$ are $\theta_i$ for $i\not\in\ii$ and $\frac\pi 2 - \theta_i$ for $i\in\ii$. \label{it:theta Wi}
	\end{enumerate}
\end{lemma}
\begin{proof}
	\emph{(\ref{it:Fi})} As $\I_i$ and $f_j$ commute if $i\neq j$, and $I_if_i = e_i^\perp$, we have $\I_\ii F = \I_{i_1}\cdots\I_{i_k} f_1\cdots f_p = f_1\cdots \I_{i_1} f_{i_1} \cdots \I_{i_k} f_{i_k} \cdots f_p = f_1\cdots e_{i_1}^\perp \cdots e_{i_k}^\perp \cdots f_p$. 
	\emph{(\ref{it:beta i}, \ref{it:theta Wi})} By the previous item, $\beta_\ii$ is a basis of $[F_\ii]$, obtained from $\beta_W$ by replacing, for $i\in\ii$, $f_i$ with $e_i^\perp$, which is in the same principal plane and makes with $e_i$ an angle $\frac\pi 2 - \theta_i$. Thus condition \eqref{eq: ei fj}, with the appropriate substitutions, is satisfied for $\beta_V$ and $\beta_\ii$.
\end{proof}

%Note that $\mathbf{i}_I =0$ if $I$ has some $i\leq d$, as $\mathbf{i}_i=0$.

\begin{theorem}\label{pr:exp Theta}
	$\displaystyle e^{\btheta_{V,W}} = \sum_{\ii\in\II^{d+1,p}} \cos\Theta_{V,W_\ii} \,\I_\ii = \sum_{\ii\in\II^{d+1,p}} \pi_{V,W_\ii} \,\I_\ii$.
\end{theorem}
\begin{proof}
	Follows from \Cref{pr:exp expansion theta_i}, since $\cos\Theta_{V,W_\ii} = \prod_{i\notin \ii} \cos\theta_i  \prod_{i\in \ii} \sin\theta_i$ by \Cref{pr:properties Grassmann}\,\emph{\ref{it:Theta prod cos}} and \Cref{pr:Ii Wi}\emph{\ref{it:theta Wi}}.
\end{proof}

With $e^{\btheta_{V,W}}$ decomposed like this, each component shows how volumes in $V$ contract when orthogonally projected on a $W_\ii$.
In particular:
%$\inner{e^{\btheta_{V,W}}}_0 = \cos\Theta_{V,W} = \pi_{V,W}$, as $W_0=W$.
%Also, $\|\inner{e^{\btheta_{V,W}}}_{2p}\| = \cos\Theta^\perp_{V,W} = \pi_{V,W^\perp}$, since if $d\neq 0$ then $\inner{e^{\btheta_{V,W}}}_{2p}=0$ and $\Theta^\perp_{V,W} = \frac\pi2$, by \Cref{pr:complementary simple cases}\emph{\ref{it:Theta perp pi 2}}, otherwise  for $\ii=(1,\ldots,p)$ we have $W_\ii = P_{W^\perp}(V)$, and so $\Theta_{V,W_\ii} = \Theta_{V,W^\perp}$, by \Cref{pr:properties Grassmann}\emph{\ref{it:Theta PWV}}.

\begin{proposition}\label{pr:0 2p components exp}
	$\inner{e^{\btheta_{V,W}}}_0 = \cos\Theta_{V,W} = \pi_{V,W}$, and also $\|\inner{e^{\btheta_{V,W}}}_{2p}\| = \cos\Theta^\perp_{V,W} = \pi_{V,W^\perp}$.
\end{proposition}
\begin{proof}
	In \Cref{pr:exp Theta}, $\inner{e^{\btheta_{V,W}}}_0$ is given by $\ii=0$, and $W_0=W$.
	If $d\neq 0$ then $\inner{e^{\btheta_{V,W}}}_{2p}=0$ and $\Theta^\perp_{V,W} = \frac\pi2$ by \Cref{pr:complementary simple cases}\emph{\ref{it:Theta perp pi 2}}.
	If $d=0$, $P_{W^\perp}(V) = W_\ii$ for $\ii=(1,\ldots,p)$, and $\Theta_{V,W_\ii} = \Theta_{V,W^\perp}$ by \Cref{pr:properties Grassmann}\emph{\ref{it:Theta PWV}}.
\end{proof}

The theorem can be adapted for oriented angles, once we orient $W_\ii$.

\begin{definition}
	For $\ii\in\II^{d+1,p}$, we give $W_\ii$ the orientation of $B_\ii = I_\ii B = \epsilon_B  \|B\| F_\ii$, and set $\epsilon_{A,B_\ii}$ in terms of $\beta_V$ and $\beta_\ii$, 
		\SELF{ordered accordingly}
\end{definition}

\begin{proposition}\label{pr:oriented exp theta}
	$\displaystyle e^{\btheta_{A,B}} = \sum_{\ii\in\II^{d+1,p}} \cos\Theta_{A,B_\ii} \,\I_\ii = \sum_{\ii\in\II^{d+1,p}} \pi_{A,B_\ii} \,\I_\ii$.
\end{proposition}
\begin{proof}
	Follows from \Cref{pr:exp theta+ exp theta} and \Cref{pr:exp Theta}, as $\epsilon_{A,B_\ii} = \epsilon_{A,B}$.
\end{proof}

\subsection{Plücker Coordinates}\label{sc:Plücker coordinates}

It will be interesting to rewrite \Cref{pr:exp Theta} in a different form. This will require some more notation.

\begin{definition}
	Extend $\beta_W$ to the orthonormal basis of $Y=V+W$ given by $\beta_Y = (f_1,\ldots,f_d,e_{d+1}^\perp,f_{d+1},\ldots,e_p^\perp,f_p) = (y_1,\ldots,y_{2p-d})$.
	Its \emph{coordinate blades} are $C_0=1$ and $C_\jj = y_{ j_1} y_{j_2} \cdots y_{ j_k}$ for $\jj=(j_1,\ldots,j_k)\in\II^{2p-d}$, forming orthonormal bases $\beta_{\bigwedge^k Y} = ( C_\jj)_{\jj\in\II^{2p-d}_k}$ and $\beta_{\bigwedge Y} = ( C_\jj)_{\jj\in\II^{2p-d}}$ of $\bigwedge^k Y$ and $\bigwedge Y$.
	Each $Y_\jj = [C_\jj]$ with $\jj\in\II^{2p-d}_k$ is a \emph{coordinate $k$-subspace}.
\end{definition}

Coordinates of a blade $A\in\bigwedge^k Y$ in $\beta_{\bigwedge^k Y}$ give homogeneous \emph{Plücker coordinates} of $[A]$ w.r.t.\! $\beta_Y$.

\begin{definition}
	Let $\sigma : \bigwedge Y \rightarrow \bigwedge Y$ be given by $\sigma(C) = C F^{-1} = C f_p \cdots f_1$ for any $C\in\bigwedge Y$.
\end{definition}

This map produces a permutation of $\beta_{\bigwedge Y}$. 
Also, $\I_\ii = e_{i_1}^\perp f_{i_1} \cdots e_{i_k}^\perp f_{i_k}$ and $F_\ii$ are elements of $\beta_{\bigwedge Y}$, with $\I_\ii = \sigma(F_\ii)$.

Among all coordinate $p$-subspaces of $\beta_Y$, the $W_\ii$'s are those having either $f_i$ or $e_i^\perp$ for each $i$, 
\SELF{They intersect each principal plane in a line. Não pode ter ambos sem faltar em outro, pois há $p$ planos}
and the projection factor of $V$ on any other vanishes.
So we can extend the sum in \Cref{pr:exp Theta} as follows.

\begin{proposition}\label{pr:geom prod coord subspaces}
	$\displaystyle e^{\btheta_{V,W}} = \!\sum_{\jj\in\II_p^{2p-d}} \cos\Theta_{V,Y_\jj} \, \sigma(C_\jj) = \!\sum_{\jj\in\II_p^{2p-d}} \pi_{V,Y_\jj} \, \sigma(C_\jj)$.
\end{proposition}
\begin{proof}
	Given $\ii \in\II^{d+1,p}$ we have $F_\ii = C_\jj$ for some $\jj\in\II_p^{2p-d}$, and so $W_\ii = Y_\jj$ and $\I_\ii = \sigma(C_\jj)$.
	And given $\jj\in\II_p^{2p-d}$, if $C_\jj$ has either $f_i$ or $e_i^\perp$
	\SELF{not both, as there are $p$ pairs}
	for each $1\leq i\leq p$ then forming $\ii\in\II^{d+1,p}$ with the indices for which it has $e_i^\perp$ we obtain $C_\jj = F_\ii$.
	Otherwise $e_i \perp Y_\jj$ for some $i$ and $\pi_{V,Y_\jj} = 0$.
	So the nonzero terms in the sums above are the same as in \Cref{pr:exp Theta}.
\end{proof}

By \Cref{pr:theta rotates V to W}, $E = \sigma^{-1}(e^{\btheta_{V,W}}) = \sum_{\jj\in\II_p^{2p-d}} \cos\Theta_{V,Y_\jj} \, C_\jj$, so the coefficients in \Cref{pr:geom prod coord subspaces} are Plücker coordinates of $V$ w.r.t.\! $\beta_Y$. 
They are normalized, with
\begin{equation}\label{eq:Pythagorean identity}
	\sum_{\jj\in\II_p^{2p-d}} \cos^2\Theta_{V,Y_\jj} = 1,
\end{equation}
but also scrambled by $\sigma$.
To relate each coefficient in $e^{\btheta_{V,W}}$ to the correct coordinate blade or subspace, note that the only ones that do not vanish are those from \Cref{pr:exp Theta}, and $\I_\ii$ corresponds to $F_\ii$ via $\sigma^{-1}$.

Though useful for theoretical purposes, Plücker coordinates are computationally expensive if $n=\dim X$ is large \cite[p.197]{Stolfi1991}. 
It takes $\binom{n}{p}$ coordinates to represent $p$-subspaces, but Plücker relations reduce the dimension of their Grassmannian to $p(n-p)$. When $p$ is not close to $1$ or $n$, blades become sparse among multivectors, and this representation becomes quite inefficient. 
There are more economical, even if less elegant, ways to locate a subspace (e.g., reduced simplex representations can use as low as $p(n-p)+1$ coordinates  \cite[p.204]{Stolfi1991}).
The angle bivector strikes a nice balance between economy and theoretical convenience.

\subsection{Distinct Dimensions and Projective-Orthogonal Decomposition}\label{sc:Subspaces of different dimensions}

Now we treat the case of different dimensions.
Let $A\in\bigwedge^p X$ and $B\in\bigwedge^q X$ be nonzero blades, $m=\min\{p,q\}$, $\beta_V$ and $\beta_W = (f_1,\ldots,f_q)$ be associated principal bases of $V=[A]$ and $W=[B]$, and $\epsilon_ B$ be as in \eqref{eq:principal decomposition}.

\begin{definition}
	A \emph{projective-orthogonal (PO) decomposition} of $B$ w.r.t.\! $A$ is $B = B_P B_\perp$, where $B_P = \epsilon_B \|B\| f_1f_2\cdots f_m$ and $B_\perp = f_{m+1}f_{m+2}\cdots f_q$ ($=1$ if $p\geq q$) are, respectively, \emph{projective and orthogonal subblades}.
	We also decompose $W = W_P\oplus W_\perp$, where $W_P=[B_P]$ and $W_\perp=[B_\perp]$ are, respectively, \emph{projective and orthogonal subspaces} of $W$ w.r.t.\! $V$.
\end{definition}

%Note that $\| B_P\|=1$ and $\| B_\perp\|=\| B\|$.
%Clearly, $[B_P] = W_P$ and $[B_\perp] = W_\perp$, as in \eqref{eq:decomposition W}.

%The orientation of $B_P$ is aligned with $A$, and $\epsilon_{A, B}$ in $B_\perp$ offsets orientation changes in $f_{p+1} f_{p+2}\ldots f_q$.
%\SELF{It is also sensitive to changes in $f_i$ with $\theta_i=\frac\pi 2$}

Note that $B_\perp$ is completely orthogonal to $A$ and $B_P$, so $B = B_P\wedge B_\perp$. 
Also, $\beta_{W_P} = (f_1,\ldots,f_m)$ is a principal basis of $W_P$ associated to $\beta_V$, for which $\epsilon_{A,B_P} = \epsilon_{A,B}$, and the principal angles of $V$ and $W_P$ are the same as those of $V$ and $W$.
If $p\leq q$ then  $\grade(B_P) = \grade(A)$.

If $p\geq q$ then $B_P=B$, $B_\perp=1$, $W_P=W$ and $W_\perp=\{0\}$. 
If $p<q$ and $V\not\pperp W$, \eqref{eq:ProjVW} gives $W_P=P_W(V)$ and $W_\perp = W\cap V^\perp$, and the blades are unique up to signs, with $B_P$ having the orientation of $\epsilon_{A,B}\P_BA$.
	\SELF{$B_P=\epsilon_{A,B}\|B\|\frac{\P_{B} A}{\|\P_{B} A\|}$ and $B_\perp = B_P^{-1}B$} 
If $p<q$ and $V\pperp W$, both decompositions depend on the choice of $\beta_W$, with $W_P \supsetneq P_W(V)$ and $W_\perp \subsetneq W\cap V^\perp$.

\begin{proposition}\label{pr:Thetas different dim}
	 $\Theta_{V,W}=\Theta_{V,W_P}$, $\Theta_{V,W}^\perp = \Theta_{V,W_P}^\perp$, $\Theta_{A,B} = \Theta_{A,B_P}$ and $\Theta_{A,B}^\perp = \Theta_{A,B_P}^\perp$ (for oriented angles w.r.t. $\beta_V$, $\beta_W$ and $\beta_{W_P}$).
\end{proposition}
\begin{proof}
	Follows from \eqref{eq:Theta prod cos} and \eqref{eq:Theta perp prod sin}.
\end{proof}

We extend \Cref{df:angle bivector,df:oriented angle bivector} to the case of distinct dimensions.

\begin{definition}
	$\btheta_{V,W} = \btheta_{V,W_P}$ and $\btheta_{A,B} = \btheta_{A,B_P}$ (w.r.t. $\beta_V$, $\beta_W$ and $\beta_{W_P}$).
\end{definition}

When $p<q$, the non-uniqueness of the decompositions can increase the ambiguity of the angle bivectors.
Not even $e^{\btheta_{A,B}}$ is uniquely defined anymore, as we can switch the orientation of $B_P$, and if $V\pperp W$ we can swap $f_p$ with any unit vector in $W_\perp$.

Still, our results readily adapt.
For example, if $p\leq q$ and $\|A\|=\|B\|=1$ then  $e^{-\frac 12 \btheta_{A,B}} A e^{\frac 12 \btheta_{A,B}} = B_P$, and so $V$ rotates onto $W_P$.
And as $B_\perp$ is completely orthogonal to all principal bivectors in $\btheta_{B,A} = - \btheta_{A,B}$, we have $e^{-\frac 12 \btheta_{B,A}} B_P B_\perp e^{\frac 12 \btheta_{B,A}} = A B_\perp$, so that $W$ rotates to $V\oplus W_\perp$.

\section{Clifford Geometric Product}\label{sc:geometric product}

We relate the geometric product of blades to the angle bivector, and  interpret geometrically some of its well known algebraic properties.
We consider first equal grades, leaving the general case for \Cref{sc:Blades of different grades}, and for completeness we define $\btheta_{A,B}=0$ if $A$ or $B$ is $0$.

%Let $ A,B\in\bigwedge^p X$ be nonzero blades, and $\beta_V = (e_1,\ldots,e_p)$ and $\beta_W = (f_1,\ldots,f_p)$ be associated principal bases of $V=[A]$ and $W=[B]$.

\begin{theorem}\label{pr:geom prod angle bivector}
	$\tilde{A}B = \|A\|\|B\| \,e^{\btheta_{A,B}}$ for same grade blades $A,B\in\bigwedge^p X$.
\end{theorem}
\begin{proof}
	For $A,B\neq 0$, \Cref{pr:oriented theta rotates A to B} gives $\frac{B}{\|B\|} = \frac{A}{\|A\|} e^{\btheta_{A,B}}$.
\end{proof}

Note that $\tilde{A}B = \epsilon_{A,B} \|A\|\|B\| \,e^{\btheta_{[A],[B]}}$ carries the relative orientation of $A$ and $B$, which makes sense as $\epsilon_{A,B}$ is the sign of $\tilde{A}*B = \inner{\tilde{A}B}_0$ (if $A*B\neq 0$).
Still, this makes relating their orientations in $AB = \epsilon_{\tilde{A},B} \|A\|\|B\| \,e^{\btheta_{[A],[B]}}$  less immediate.
%, even impossible if $p$ is unknown.
In \Cref{sc:Products with reversions} we discuss how this affects other products.

With \Cref{pr:oriented exp theta} we obtain
\begin{equation}\label{eq:AB projections}
	\tilde{A}B = \|A\|\|B\|  \sum_{\ii\in\II^{d+1,p}} \pi_{A,B_\ii} \,\I_\ii 
	= \epsilon_{A,B} \sum_{\ii\in\II^{d+1,p}} \|P_{B_\ii}A\|\|B\|  \,\I_\ii.
\end{equation}
The projections from $A$ to the $B_\ii$'s make $A$ and $B$ seem to play very different roles in the product.
But $A_\ii = A\I_\ii$ satisfies $\pi_{A,B_\ii}=\pi_{B,A_\ii}$, as $e_i \cdot e_i^\perp = f_i \cdot f_i^\perp$, and so $\tilde{A}B = \epsilon_{A,B} \sum_{\ii} \|A\|\|P_{A_\ii}B\|  \,\I_\ii$ as well.
\SELF{$= \frac\pi 2 - \theta_i$}

With $\tilde{A}B$ decomposed  in terms of products of principal bivectors oriented from $[A]$ to $[B]$, as above, all coefficients have the same sign $\epsilon_{A,B}$.
\SELF{If $A\pperp_p B$ we can change the sign of $\I_p$, and $\epsilon_{A,B}$ changes to compensate. Terms without $\I_p$ vanish.}
The only effect of swapping $A$ and $B$ (of same grade) is to reorient principal planes, from $[B]$ to $[A]$.
Applying a reversion to \eqref{eq:AB projections} we obtain $\tilde{B}A = \epsilon_{A,B} \sum_{\ii} \|P_{B_\ii}A\|\|B\| \,\tilde{\I}_\ii$,
so that components change sign depending on whether $\tilde{\I}_\ii = \pm \I_\ii$ has an even or odd number of $I_i$'s.
\SELF{$\I_i=0$ for $i\leq d$, otherwise $\tilde{\I}_i=-\I_i$}
In \Cref{sc:Commutator and anticommutator} we relate this to the commutator of blades.

\begin{example}\label{ex:geom product 1a}
	Let $\{f_1,f_2,g_1,g_2\}$ be orthonormal, $e_1=\frac{f_1+3g_1}{\sqrt{10}}$, $e_2=\frac{2f_2+g_2}{\sqrt{5}}$, $A=e_1 e_2$ and $B=f_1 f_2$.
	Then $\beta_A=(e_1,e_2)$ and $\beta_B=(f_1,f_2)$ are associated principal bases of $[A]$ and $[B]$, with $\I_i=g_i f_i$ and $e_i^\perp=g_i$.
	As $\epsilon_{A,B}=1$, all coefficients in $\tilde{A}B = (2+6\I_1+\I_2+3\I_1\I_2)/5\sqrt{2}$ are positive.
	And since $\|A\|=\|B\|=1$, the coefficients are, in order, projection factors of $[A]$ on $[B]=[f_1f_2]$, $[e_1^\perp f_2]$, $[f_1 e_2^\perp]$ and $[e_1^\perp e_2^\perp] = ([B]^\perp)_P$ (the projective subspace of $[B]^\perp$ w.r.t.\! $[A]$). 
	
	In $\tilde{B}A = (2+6\tilde{\I}_1+\tilde{\I}_2+3\tilde{\I}_1\tilde{\I}_2)/5\sqrt{2} = (2-6\I_1-\I_2 +3\I_1\I_2)/5\sqrt{2}$, terms with a single principal bivector switch signs. Projection factors have the same values as before, but now refer to projections from $[B]$ to 
	$[A]=[e_1 e_2]$, $[f_1^\perp e_2]$, $[e_1 f_2^\perp]$ and $[f_1^\perp f_2^\perp] = ([A]^\perp)_P$, where $f_1^\perp = \frac{3f_1-g_1}{\sqrt{10}}$, $f_2^\perp = \frac{f_2-2g_2}{\sqrt{5}}$ and $([A]^\perp)_P$ is the projective subspace of $[A]^\perp$ w.r.t.\! $[B]$.
\end{example}

\subsection{Plücker Coordinates in the Geometric Product}

Consider \Cref{pr:geom prod coord subspaces} with $V=[A]$ and $W=[B]$.
%Extend $\beta_W$ to the basis $\beta_Y$ of $Y=V+W$ as in \Cref{sc:Plücker coordinates}, and again let the $C_\jj$'s be its coordinate $p$-blades and $Y_\jj = [C_\jj]$. 
As the coefficients in that decomposition give Plücker coordinates of $V$ w.r.t.\! $\beta_Y$, and these are homogeneous, the same holds for the coefficients in
\begin{equation}\label{eq:AB Yj}
	\tilde{A}B =  \epsilon_{A,B} \|A\|\|B\| \sum_{\jj\in\II_p^{2p-d}} \cos\Theta_{V,Y_\jj} \, \sigma(C_\jj).
\end{equation}

\begin{example}
	In \Cref{ex:geom product 1a}, $[A]$ has Plücker coordinates $(2\!:\!6\!:\!1\!:\!3\!:\!0\!:\!0)$ in the basis $(f_1f_2, e_1^\perp f_2, f_1 e_2^\perp, e_1^\perp e_2^\perp, f_1 e_1^\perp, f_2 e_2^\perp)$ of $\bigwedge^2Y$ obtained from $\beta$.
	The last coordinates vanish as $A$ is partially orthogonal to any coordinate blade having neither $e^\perp_i$ nor $f_i$ for some $i$.
	%	$A \pperp f_1 e_1^\perp$ and $A \pperp f_2 e_2^\perp$.
	With $\tilde{B}A$ in terms of $\tilde{\I}_i$'s, we find the same coordinates for $[B]$ in the basis $(e_1e_2, f_1^\perp e_2, e_1 f_2^\perp, f_1^\perp f_2^\perp, e_1f_1^\perp, e_2f_2^\perp)$ %of $\bigwedge^2([A]+[B])$ 
	obtained from $\{e_1,e_2,f_1^\perp,f_2^\perp\}$.
\end{example}

\begin{example}\label{ex:geom product 2}
	In Fig. \ref{fig:geometric-product}, $[A]$ has normalized Plücker coordinates $(\frac 3 5:\frac 4 5:0)$ w.r.t.\! $(f_1 f_2, f_1 g_2,f_2 g_2)$, and $B=f_1 f_2$.
	From this, the norms and orientations, we obtain $\tilde{A}B = -3 -4\I_2$, where $\I_2= g_2 f_2$ is one of the principal bivectors, oriented from $A$ to $B$. 
	As the first principal plane is degenerate, its principal bivector is $\I_1=0$, so $\tilde{A}B$ has no $4$-vector (this is another way to look at the usual result that $[A]\cap[B]\neq\{0\} \Rightarrow A\wedge B=0$).
	
	\begin{figure}
		\centering
		\includegraphics[width=0.45\linewidth]{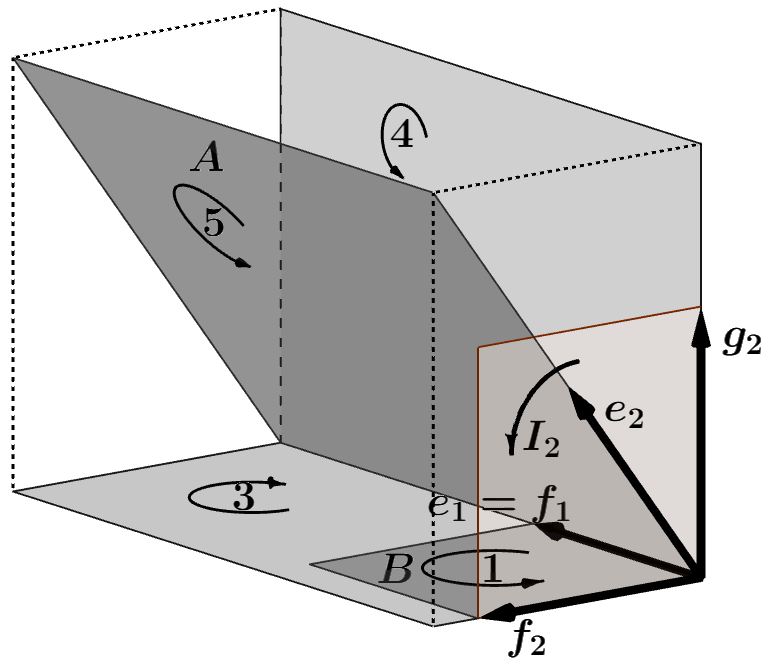}
		\caption{Orthogonal projections of a blade $A$ on coordinate planes of the orthonormal basis $(f_1,f_2,g_2)$. Arcs show orientations, and the numbers inside them are blade areas.}
		\label{fig:geometric-product}
	\end{figure}
\end{example}

The lack of economy of Plücker coordinates is reflected in the geometric product of high grade blades in even larger spaces, which tends to produce a large number of components, linked by many relations.
It is an interesting question whether this is a price to be paid for the algebraic simplicity of this product: if a more economical operation could accomplish the same tasks, would its algebra necessarily be more complicated?

\subsection{Distinct Grades}\label{sc:Blades of different grades}

We can adapt \Cref{pr:geom prod angle bivector} for blades of different grades using PO decompositions. 
In the following result, the same principal bases must be used for the decomposition and to form $\btheta_{A,B}$. 

\begin{proposition}
	Let $A\in\bigwedge^p X$ and $B\in\bigwedge^q X$ be blades.
	\begin{enumerate}[i)]
		\item If $p\leq q$ then $\tilde{A} B = \|A\|\|B\| e^{\btheta_{A,B}} B_\perp$, where $B_\perp$ is an orthogonal subblade of $B$ w.r.t. $A$. \label{it:B larger}
		\item If $p\geq q$ then $\tilde{A} B = \|A\|\|B\| e^{\btheta_{A,B}} \widetilde{A_\perp}$, where $A_\perp$ is an orthogonal subblade of $A$ w.r.t. $B$. \label{it:A larger}
	\end{enumerate}
\end{proposition}
\begin{proof}
	\emph{(\ref{it:B larger})} Follows from \Cref{pr:geom prod angle bivector}, since $\tilde{A}B = (\tilde{A}\,B_P) B_\perp$, $B_P$ has the same grade as $A$, $\|B_P\|=\|B\|$ and $\btheta_{A,B_P} = \btheta_{A,B}$.
	\emph{(\ref{it:A larger})} Similar, using $\tilde{A}B = \widetilde{A_\perp}(\widetilde{A_P}\,B)$ and also that $\widetilde{A_\perp}$ commutes with $e^{\btheta_{A,B}}$, as it is completely orthogonal to all principal bivectors  .
\end{proof}

As the orthogonal subblade is completely orthogonal to all components of $e^{\btheta_{A,B}}$, we actually have $e^{\btheta_{A,B}} \wedge B_\perp$ and $e^{\btheta_{A,B}} \wedge \widetilde{A_\perp}$ in these formulas. 
So, in the geometric product of blades, the smaller one operates only with the projective subblade of the larger blade, preserving the orthogonal subblade.

This result also shows that, while $e^{\btheta_{A,B}}$ and $B_\perp$ (or $A_\perp$) can depend on the choice of principal bases, their product cannot.

\begin{example}\label{ex:geom product 3}
	Let $\{f_1,\ldots,f_4,g_1,g_2\}$ be orthonormal, $e_1=\frac{\sqrt{3}}{2} f_1+\frac 12 g_1$, $e_2=g_2$, $A=e_1 e_2$ and $B=f_1f_2f_3f_4$. 
	Then $\I_1=g_1 f_1$, $\I_2=g_2 f_2$, $\btheta_{A,B} = \btheta_{A,f_1f_2} = \frac\pi6 I_1 + \frac\pi2 I_2$, $e^{\btheta_{A,B}} = \frac{\sqrt{3}}{2}I_2+\frac12\I_1\I_2$ and $\tilde{A}B = \big(\frac{\sqrt{3}}{2}+\frac12\I_1\big)\I_2 f_3f_4$. 
	The common factor $\I_2$ is due to $\theta_2=\frac\pi2$ (as $e_2\perp f_2$, projections of $A$ on coordinate subspaces having $f_2$ instead of $e_2^\perp = g_2$ vanish), and the rest is $B_\perp=f_3f_4$. 
	The coefficients are projection factors of $[A]$ on $[f_1g_2f_3f_4]$ and $[g_1g_2f_3f_4]$, and $[A]$ has Plücker coordinates $(0\!:\!\sqrt{3}\!:\!0\!:\!1\!:\!0\!:\!0)$ in the basis $(f_1f_2,f_1g_2,g_1f_2,g_1g_2,f_1g_1,f_2g_2)$ of $\bigwedge^2[f_1f_2g_1g_2]$.
	Signs in $\tilde{B}A = \big(\frac{\sqrt{3}}{2}-\frac12\I_1\big)\I_2 f_3f_4$ are the result of reverting the $\I_i$'s and $B_\perp$.
\end{example}

\subsection{Principal Angles via Geometric Algebra}\label{sc:Principal angles via geometric algebra}

Hitzer's method \cite{Hitzer2010a} to find principal angles via geometric product can be used to decompose $e^{\btheta_{V,W}}$ or $\tilde{A}B$ in terms of $I_\ii$'s, as follows.

Given subspaces $V,W\subset X$ with $\dim V = p\leq q=\dim W$, compute $\tilde{A}B$ for unit blades $A\in\bigwedge^p V$ and $B\in\bigwedge^q W$ with $\epsilon_{A,B}=1$.
The result is a $(p+q)$-blade if, and only if, all principal angles are $\frac\pi 2$,  in which case any orthonormal bases are associated principal bases.

Assume otherwise, and let $d = \dim (V\cap W)$ and $D = \max\{i:\theta_i\neq \frac\pi 2\}$, whose values we might not know yet. Using \Cref{pr:exp product rotors}, with $R_i=1$ for $i\leq d$, $R_i=I_i$ for $i>D$, and expanding the other $R_i$'s as in \Cref{pr:exp expansion theta_i}, we find that the nonzero components of $\tilde{A}B = e^{\btheta_{A,B}} B_\perp$ are:
\begin{subequations}\label{eq:simplified expansion}
	\begin{align}
		\tilde{A}B &= R_{d+1}\cdots R_D \I_{D+1} \cdots\I_p B_\perp \label{eq:simplifies Rs}\\
		&= \Bigl( c_{d+1}\cdots c_D \label{eq:lowest}\\
		&\quad + s_{d+1}c_{d+2}\cdots c_D \I_{d+1} + \cdots +c_{d+1}\cdots c_{D-1}s_D \I_D \label{eq:second lowest}\\
		&\quad \hspace{4pt} \vdots \nonumber\\
		&\quad + c_{d+1}s_{d+2}\cdots s_D I_{d+2} \cdots I_D + \cdots + s_{d+1}\cdots s_{D-1} c_D \I_{d+1} \cdots I_{D-1} \nonumber \\
		&\quad + s_{d+1}\cdots s_D \I_{d+1} \cdots I_D \Bigr)\, \I_{D+1} \cdots\I_p B_\perp.  \nonumber
	\end{align}
\end{subequations}

So the lowest and highest non-vanishing grades are $p+q-2D$ and $p+q-2d$, and this tells us which $\theta_i$'s are $0$ or $\frac\pi 2$. From orthonormal bases of $V\cap W^\perp$ and $W\cap V^\perp$ we get principal vectors $e_{D+1},\ldots,e_p$ and $f_{D+1},\ldots,f_q$ to form $I_{D+1},\ldots,I_p$ and $B_\perp$.

Also, the product of the non-zero components of second lowest grade \eqref{eq:second lowest} by the inverse of the lowest grade blade \eqref{eq:lowest} is a bivector. Decomposing it into commuting blades we obtain $\tan\theta_{d+1} I_{d+1}+\cdots+\tan\theta_{D} I_{D}$ and find $\theta_i$ and $I_i$ for $d<i\leq D$.

With this we can write the decompositions.
If desired, we can also obtain principal vectors $e_i=f_i$ for $i\leq d$ from an orthonormal basis of $V\cap W$, and for $d<i\leq D$ take unit vectors $e_i \in V\cap[I_i]$ and $f_i \in W\cap [I_i]$ with $e_i\cdot f_i >0$.

\subsection{The Hidden Geometry of Algebraic Properties}

In this section, we use our results to reveal the rich geometry that lies behind some simple and well known algebraic properties of the Clifford product.

\subsubsection{Invertibility}

The invertibility of non-null blades $A\in\bigwedge^p X$ and $B\in\bigwedge^q X$ can be seen as the result of $AB$ carrying all geometric data needed to, given one blade, recover the other one.

If $p=q$ we can see in $A = AB \frac{\tilde{B}}{\|B\|^2} = \frac{\tilde{A}B}{\|B\|^2} B = \bigl(\epsilon_{A,B} \frac{\|A\|}{\|B\|} \,e^{-\btheta_{[B],[A]}} \bigr) B$ how $AB$ has all we need to change the orientation, norm and subspace (via \Cref{pr:theta rotates V to W}) of $B$ into those of $A$.
More precisely, the Plücker coordinates stored in $AB = \sum_\ii \epsilon_{\tilde{A},B} \|A\|\|B\|\pi_{A,B_\ii}  \,\I_\ii$ allow us to locate $[A]$ relative to $[B]$. Each component has all information needed to, using $B$, find one of $A = \sum_{\jj\in\II_p^{2p-d}} P_{C_\jj} A = \sum_{\ii\in\II^{d+1,p}} P_{B_\ii} A$:
\begin{align*}
	\left(\epsilon_{\tilde{A},B} \|A\|\|B\|\pi_{A,B_\ii} \,\I_\ii\right)\big(\tilde{B}/\|B\|^2\big) 
	&= \epsilon_{A,B} \|P_{B_\ii}A\| \,\I_\ii \, B/\|B\| \\
	&= \epsilon_{A,B} \|P_{B_\ii}A\| \,\I_\ii \,\epsilon_B F \\
	&= \epsilon_A \|P_{B_\ii}A\| \,F_\ii \\
	&= P_{B_\ii} A.
\end{align*}
If $p<q$, we obtain the same result with components of $AB = (AB_P)B_\perp$. 
If $p>q$, each component of $AB = (-1)^{p(p-q)}A_\perp(A_P B)$ multiplied by $B^{-1}$ gives $(P_{B_\ii}A_P)A_\perp$.
\SELF{$(-1)^{p(p-q)}A_\perp A_P B B^{-1} = (-1)^{p(p-q)}A_\perp \sum (P_{B_\ii}A_P) = \sum (P_{B_\ii}A_P) A_\perp$}

In particular, inverting $B$ means finding $A$ such that $AB=1$. This $1$ may seem to carry too little information, but it has all we need. The fact that it is a scalar means $AB$ has no orthogonal subblade, so $p=q$, and no $I_\ii$'s, so $[A]=[B]$ as all other Plücker coordinates vanish. Its norm implies $\|A\|=1/\|B\|$, and its sign shows $\epsilon_{\tilde{A},B} = 1$, so $A$ has the orientation of $\tilde{B}$. Putting it all together we find $A = \tilde{B}/\|B\|^2$, as expected.

\subsubsection{The Geometry of $\|AB\|=\|A\|\|B\|$}

The relation \eqref{eq:AB Yj} between the geometric product and the asymmetric angles is more complicated than those we give in \Cref{sc:Subproducts and Grassmann angles} for other products.
This is understandable, since this product includes the others as its components, and carries information about projections on various subspaces. 

Surprisingly, this complexity is behind one of its simplest properties: for blades, $\|AB\|=\|A\|\|B\|$.  The algebraic proof  is deceivingly easy, but not very illuminating, and does not explain what makes this product special in this respect, while others are submultiplicative. 

A geometric proof can use \eqref{eq:AB Yj} and \eqref{eq:Pythagorean identity} to obtain, for same grade blades, $\|AB\|^2 = \|A\|^2\|B\|^2 \sum_\jj \cos^2\Theta_{V,Y_\jj} = \|A\|^2\|B\|^2$.
And if $B$, for example, has larger grade, using a PO decomposition, and since $B_\perp$ is completely orthogonal to $A$ and $B_P$, we also find
%as $B_\perp$ is completely orthogonal to $A$ and $B_P$ we have 
$\|AB\| = \|AB_P\| \|B_\perp\| = \|A\|\|B_P\| = \|A\|\|B\|$.

To show what is behind this, we note that \eqref{eq:Pythagorean identity} is in fact a general identity for asymmetric angles with coordinate $p$-subspaces of orthogonal bases \cite{Mandolesi_Grassmann}, that leads to a generalized \Pythagorean\ theorem \cite{Mandolesi_Pythagorean} stating that projections on all such subspaces preserve the total squared volume\footnote{In complex spaces, the (non-squared) volume is the sum of volumes of projections \cite{Mandolesi_Pythagorean}.}, i.e., $\|A\|^2 = \sum_{\jj} \|P_{Y_\jj}A\|^2$.
Writing \eqref{eq:AB Yj} as 
$\tilde{A}B = \epsilon_{A,B} \|B\| \sum_{\jj} \|P_{Y_\jj}A\| \, \sigma(C_\jj)$,
we again find $\|AB\|^2 = \|B\|^2 \sum_\jj \|P_{Y_\jj}A\|^2 = \|A\|^2\|B\|^2$.

So, while the products in \eqref{eq:subproducts} are submultiplicative for blades because they involve projections on single subspaces, which shrink volumes and thus reduce norms, the geometric product preserves norms precisely because it involves projections on all coordinate $p$-subspaces $Y_\jj$.

\subsubsection{Duality}\label{sc:Duality}

The relation\footnote{We adopt the convention that $^*$ takes precedence over the product, so $AB^*$ means $A(B^*)$.} $(AB)^* = AB^*$, where $^*$ denotes the dual obtained via product with $J^{-1}$ for a given unit pseudoscalar $J$, is algebraically trivial, a mere consequence of the associativity of the geometric product.
But it expresses a duality between $AB$ and $AB^*$ which, as we show, is linked to another one between $e^{\btheta_{V,W}}$ and $e^{\btheta_{V,W^\perp}}$, reflecting a symmetry  that swaps sines and cosines in the components of \Cref{pr:exp expansion theta_i}. 

Let $A\in\bigwedge^p X$ and $B\in\bigwedge^q X$ be unit blades, $V=[A]$ and $W=[B]$, with associated principal bases $\beta_V=(e_1,\ldots,e_p)$ and $\beta_W=(f_1,\ldots,f_q)$, and  principal angles $\theta_1,\ldots,\theta_m$ for $m=\min\{p,q\}$, and let $e_i^\perp$, $I_i$, $R_i$ be as before.

We consider first a case with $p=q$ and $V\cap W=\{0\}$, and take duals w.r.t.\! $J=I_1I_2\cdots I_p$ (for which $[J] = V\oplus W$).
Completing $(e_1^\perp,\ldots,e_p^\perp)$ to a principal basis of $W^\perp$ w.r.t.\! $V$, we obtain $\btheta_{V,W^\perp} = \sum_{i=1}^p (\frac\pi2-\theta_i) \tilde{I_i}$ and principal rotors $R_i^\perp = e_ie^\perp_i = R_i \tilde{I}_i = \sin\theta_i + \tilde{I}_i \cos\theta_i$.

\begin{proposition}
	Under the above conditions	we have $(e^{\btheta_{V,W}})^* = e^{\btheta_{V,W^\perp}}$ and $(e^{\btheta_{A,B}})^* = e^{\btheta_{A,B^*}}$.
\end{proposition}
\begin{proof}
	\Cref{pr:exp product rotors} gives $(e^{\btheta_{V,W}})^* = R_1\cdots R_p \tilde{I_p}\cdots \tilde{I_1} = R_1^\perp\cdots R_p^\perp = e^{\btheta_{V,W^\perp}}$. 
		\SELF{$J^{-1} = \tilde{I_p}\cdots \tilde{I_1}$}
	And using \eqref{eq:principal decomposition} we find $B^* = \epsilon_ B f_1 f_2\cdots f_p \tilde{I_p}\cdots \tilde{I_1} = \epsilon_B e_1^\perp\cdots e_p^\perp$, so $\epsilon_{A,B^*} = \epsilon_{A,B}$ and $[B^*]$ is the projective subspace of $W^\perp$ w.r.t.\! $V$.  Thus the second identity follows from the first one and \Cref{pr:exp theta+ exp theta}.
\end{proof}

\Cref{pr:geom prod angle bivector} shows, at least under the above conditions, that this duality of exponentials lies behind $(AB)^* = AB^*$.
Even better, expanding principal rotors as in \Cref{pr:exp expansion theta_i}, we observe a component-wise duality,
\begin{align*}
	(e^{\btheta_{V,W}})^* 
	&= (c_1\cdots c_p)^* + (s_1c_2\cdots c_p I_1)^* + \cdots + (s_1\cdots s_p I_1\cdots I_p)^* \\
	&= c_1\cdots c_p \tilde{I_1}\cdots \tilde{I_p} + s_1c_2\cdots c_p\tilde{I_2}\cdots \tilde{I_p} + \cdots + s_1\cdots s_p = e^{\btheta_{V,W^\perp}},
\end{align*}
so each component of grade $k$ in $AB$ is dual to one of grade $2p-k$ in $AB^*$, generalizing the dualities \cite[p.\,82]{Dorst2007} $(A\wedge B)^* = A\glcontr B^*$ and $(A\glcontr B)^* = A\wedge B^*$ between outer product and contraction\footnote{See \Cref{df:subproducts}.}
(or Hestenes inner product \cite[p.\,23]{Hestenes1984clifford}).

In the general case ($p\neq q$, $V\cap W\neq\{0\}$, duals w.r.t.\! the whole space), the duality between exponentials becomes more complicated, but it still has the same kind of symmetry, even if the corresponding grades are different. Taking duals in \eqref{eq:simplifies Rs} w.r.t. $J=f_1\cdots f_d I_{d+1}\cdots I_p f_{p+1}\cdots f_q$ (for which $[J]=V+W$) we find $(\tilde{A}B)^* = (R_{d+1}\cdots R_D \I_{D+1} \cdots\I_p f_{p+1}\cdots f_q)^* = R_D^\perp\cdots R_{d+1}^\perp f_d\cdots f_1$, and likewise for $\tilde{A}B^*$.
\SELF{Ignoring signs, 
	$B^*=e^\perp_{d+1}\cdots e^\perp_p$, and as $R^\perp_i= e_ie_i^\perp$ if $i>d$ ($=1$ if $i>D$) we have
	$AB^* = e_1\cdots e_d R^\perp_{d+1}\cdots R^\perp_D$, and $e_i=f_i$ if $i\leq d$)}
So, except for the extra blades shifting the grade correspondence, the duality between $AB$ and $AB^*$ is again due to $R_i$'s turning into $R_i^\perp$'s. 
Taking duals w.r.t.\! the whole space, another grade shift makes each component of grade $k$ in $AB$ dual to one of grade $n-k$ in $AB^*$, where $n=\dim X$.

\section{Other Geometric Algebra Products}\label{sc:Other products}
%\section{Component subproducts}\label{sc:Component subproducts}

Let $A\in\bigwedge^p X$ and $B\in\bigwedge^q X$ be blades. 
As is known, $AB$ can have nonzero components of grades $|q-p|, |q-p|+2,\ldots,p+q$, with the first and last ones (and also some of the vanishing ones) giving useful \emph{component subproducts}\footnote{The geometric product can be defined axiomatically and the other products obtained as its components \cite{Dorst2002,Hestenes1984clifford,Hitzer2012}, or it can be defined using Grassmann algebra products, taken as more fundamental ones \cite{Dorst2007,Gull1993imaginary}. The different approaches are discussed in \cite{Lounesto2001}.}.
Most of these products are well known, and we have already been using some, but we provide here a definition for easy reference (for general multivectors, they are extended linearly):

\begin{definition}\label{df:subproducts}
	The \emph{scalar product} $A*B$, \emph{left and right contractions}\footnote{These are Dorst's symbols \cite{Dorst2002}. Lounesto \cite{Lounesto1993} uses $\lrcorner$ and $\llcorner$, which we will reserve for a slightly different contraction in \Cref{sc:Products with reversions} and \Cref{sc:Grassmann products}.} 
	$A\glcontr B$ and $A\grcontr B$, \emph{(fat) dot product} $A\bullet B$, and \emph{outer product} $A\wedge B$ are, respectively, the components of grades $0$, $q-p$, $p-q$, $|q-p|$ and $p+q$ of $AB$ (a negative grade means the component is $0$).
	\emph{Hestenes inner product} $A\cdot B$ is the $|q-p|$ component if $p,q\neq 0$, vanishing otherwise\footnote{This is the definition from \cite{Hestenes1984clifford}. In later works \cite{Hestenes2005}, Hestenes removes the exceptionality of the scalar case, so that $A\cdot B = A\bullet B$ for all $p,q\geq 0$.}.
\end{definition} 

Contractions and fat dot product, introduced by Lounesto \cite{Lounesto1993} and Dorst \cite{Dorst2002}, are less known alternatives to Hestenes inner product. These products are related by
\begin{equation}\label{eq:dots contractions}
	A\bullet B = \begin{cases}
		A\glcontr B \ \text{ if } p\leq q, \\
		A\grcontr B \ \text{ if } p\geq q,
	\end{cases}
	\ \text{ and }\quad 
	A\cdot B = 
	\begin{cases}
		A\bullet B \ \text{ if } p,q\neq 0, \\
		0 \hspace{23pt} \text{ otherwise.}
	\end{cases}
\end{equation}
Left and right contractions are related by $A\glcontr B = (\tilde{B}\grcontr\tilde{A})^\sim$, and using a PO decomposition we obtain $A\glcontr B = (A*B_P) B_\perp$.
When grades are distinct, they are asymmetric (in general, $A\glcontr B \neq B\glcontr A$ and $A\grcontr B \neq B\grcontr A$), with $A\glcontr B = 0$ if $p>q$, and $A\grcontr B = 0$ if $p<q$.
%\begin{equation*}
%	A\glcontr B = 
%	\begin{cases}
%		(A*B_P) B_\perp \ \text{ if } p\leq q, \\
%		0 \hspace{52pt} \text{ if } p>q,
%	\end{cases}
%	\text{ and }\ \   
%	A\grcontr B = 
%	\begin{cases}
%		0 \hspace{33pt} \text{ if } p<q, \\
%		(\tilde{B}\glcontr\tilde{A})^\sim \ \text{ if } p\geq q.
%%		(-1)^{q(p-q)}(A_P*B) A_\perp \ \text{ if } p\geq q.
%	\end{cases}
%\end{equation*}

As noted by Lounesto \cite[p.\,291]{Lounesto2001}, contractions have better properties than $A\cdot B$ (or $A\bullet B$).
Dorst \cite{Dorst2001,Dorst2002} also advocates for their use, arguing that identities involving $A\cdot B$ often depend on grade conditionals, which tend to accumulate as they are combined, while with contractions ``known results are simultaneously generalized and more simply expressible, without conditional exceptions'' \cite[p.\,136]{Dorst2001}, since their asymmetry allows them to `switch off' automatically when conditions fail to hold.

In \Cref{sc:Subproducts and Grassmann angles} we relate these subproducts to our angles. \Cref{sc:Commutator and anticommutator} discusses other products: the usual commutator and an anticommutator.

\subsection{Component Subproducts}\label{sc:Subproducts and Grassmann angles}

There is a reason why only certain components of the geometric product give interesting subproducts.
If $p=q$, rewriting \eqref{eq:AB projections} as
\begin{equation}\label{eq:AB Theta}
	AB = \epsilon_{\tilde{A},B}\|A\|\|B\|  \sum_{\ii\in\II^{d+1,p}} \cos\Theta_{[A],[B_\ii]} \,\I_\ii,
\end{equation}
we see how the components of grades $0,2,4,\ldots,2p$ (in fact, at most $2(p-d)$, where $d=\dim([A]\cap [B])$) are formed. We can also understand why $\inner{AB}_0$ and $\inner{AB}_{2p}$ are the most relevant ones, giving the products of \Cref{df:subproducts}: they describe projections on $[B]$ and $[B]^\perp$, while other components involve less important subspaces.
If $p\neq q$, the orthogonal subblade increases grades by $|q-p|$, so now $\inner{AB}_{|q-p|}$ and $\inner{AB}_{p+q}$ are most relevant.
% (that grades $0$ and $-|q-p|$ vanish is also useful). 

Identifying the appropriate components in \eqref{eq:AB Theta}, we obtain formulas relating the subproducts to our various angles.

\begin{theorem}\label{pr:products angles} 
	For any blades $A,B\in\bigwedge X$, 
	\begin{subequations}\label{eq:subproducts}
		\begin{align}
			|A*B| &= \|A\|\|B\|\cos\hat{\Theta}_{[A],[B]}, \label{eq:star prod Theta}\\
			\|A\glcontr B\| &= \|A\|\|B\|\cos\Theta_{[A],[B]}, \label{eq:geom lcontr Theta} \\
			\|A\grcontr B\| &= \|A\|\|B\|\cos\Theta_{[B],[A]}, \label{eq:geom rcontr Theta}\\
			\|A\bullet B\| &= \|A\|\|B\|\cos\check{\Theta}_{[A],[B]}, \label{eq:fat dot symmetrized angle} \\
			\|A\cdot B\| &= \|A\|\|B\|\cos\check{\Theta}_{[A],[B]} \quad (A, B \text{ non-scalars}), \label{eq:Hestenes symmetrized angle} \\
			\|A\wedge B\| &= \|A\|\|B\|\cos\Theta_{[A],[B]}^\perp. \label{eq:outer prod Theta perp}
		\end{align}
	\end{subequations}
\end{theorem}
\begin{proof}
	Let $p=\grade(A)$ and $q=\grade(B)$. 
	If $p=q$, the first 5 products are the $\ii=0$ component in \eqref{eq:AB Theta}, with angle $\Theta_{[A],[B]} = \Theta_{[B],[A]} = \hat{\Theta}_{[A],[B]} = \check{\Theta}_{[A],[B]}$, by \Cref{pr:properties Grassmann}\,\emph{\ref{it:Theta star}}. 
	If, moreover, d=0, then $A\wedge B$ is the $\ii=(1,\ldots,p)$ component, for which $\Theta_{[A],[B_\ii]} = \Theta_{[A],[B]^\perp}$, otherwise $A\wedge B = 0$ and $\Theta_{[A],[B]^\perp} = \frac\pi2$, by \Cref{pr:complementary simple cases}\,\emph{\ref{it:Theta perp pi 2}}.  
	
	If $p\neq q$ then $A*B=0$ and $\hat{\Theta}_{[A],[B]} = \frac\pi2$, so \eqref{eq:star prod Theta} holds.
	Using a PO decomposition of the larger blade, the unit orthogonal subblade increases grades but preserves norms. So if $p<q$ we have $\|\inner{AB}_{r+q-p}\| = \|\inner{AB_P}_r\|$, and taking $r=0$ and $r=2p$ we find, using \Cref{pr:Thetas different dim}, that \eqref{eq:geom lcontr Theta} and \eqref{eq:outer prod Theta perp} remain valid.
	Also, in this case both sides of \eqref{eq:geom rcontr Theta} vanish.
	The case $p>q$ is similar, and \eqref{eq:fat dot symmetrized angle} and \eqref{eq:Hestenes symmetrized angle} follow from \eqref{eq:dots contractions}.
\end{proof}

The geometric algebra literature has analogous results for the angles of \Cref{sc:GA angles}.
Hestenes angle definition is similar to \eqref{eq:star prod Theta}.
Hitzer  also gives it, 
\SELF{but it should have $|\cdot|$, as his angle is in $[0,\frac\pi 2]$.}
and a result like \eqref{eq:outer prod Theta perp} but with the product of sines of principal angles, which is not interpreted in terms of a single angle.
Dorst's description of contraction norm for $p\leq q$ corresponds to \eqref{eq:geom lcontr Theta}.
\SELF{quase dá a fórmula na p. 596}

Dorst's contention against $A\bullet B$ and $A\cdot B$ is supported by \eqref{eq:fat dot symmetrized angle} and \eqref{eq:Hestenes symmetrized angle}, as the min-symmetrized angle has worse properties.
Contractions, on the other hand, are related to the asymmetric angle, and their asymmetries perfectly match each other, with both sides of \eqref{eq:geom lcontr Theta} and \eqref{eq:geom rcontr Theta} vanishing as appropriate.
Dorst's defense of them agrees with our experience in using this angle: its asymmetry allows simpler proofs and more general statements, in which special cases are handled automatically, without us having to keep track of grades or dimensions.
For example, with the (symmetric) Hestenes inner product even the following simple result would be hindered by grade conditionals, requiring $0<\grade(A) \leq \grade(B)$.

\begin{corollary}\label{pr:contraction pperp}
	For any blades $A,B\in\bigwedge X$, $\| A\glcontr B\| = \|P_{B} A\|\| B\|$. Also, $A\glcontr B=0 \Leftrightarrow A\pperp B$.
	%	\begin{enumerate}[i)]
	%		\item $\| A\glcontr B\| = \|P_{B} A\|\| B\|$.
	%		\item $A\glcontr B=0 \Leftrightarrow A\pperp B$.
	%	\end{enumerate}
\end{corollary}

The symmetry of $\Theta^\perp$ is reflected in \eqref{eq:outer prod Theta perp}, which, ironically,  depends on the asymmetry of $\Theta$.
%, due to the use of \Cref{pr:properties Grassmann}\emph{\ref{it:Theta Proj}}.
%\Cref{pr:complementary simple cases}\emph{\ref{it:Theta perp pi 2}} when $[A]\cap[B]\neq\{0\}$.
For example, $A\wedge B=0$ for any $A,B\in\bigwedge^2\R^3$, but $\Theta_{[A],[B]^\perp}$ could assume any value if it were the usual (symmetric) angle between a plane $[A]$ and a line $[B]^\perp$.
Without asymmetry, \eqref{eq:outer prod Theta perp} would need the hypothesis $[A]\cap[B]=\{0\}$, as in analogous results relating volumes of parallelotopes \cite{Afriat1957} and matrix volumes \cite{Miao1992} to products of sines of principal angles.

For $v,w\in X$, \eqref{eq:outer prod Theta perp} gives the usual $\|v\wedge w\| = \|v\|\|w\|\sin\theta_{v,w}$, as \eqref{eq:Theta perp prod sin} reduces to a single sine.
We could put the formula for $\|A\wedge B\|$ in this familiar form using the sine of an angle $\Theta'_{V,W} = \frac \pi 2-\Theta^\perp_{V,W}$, which however does not have a nice interpretation in $\bigwedge X$ as $\Theta^\perp_{V,W}$ does \cite{Mandolesi_Grassmann}.

For simplicity, \Cref{pr:products angles} presented just the product norms. Now we give more detailed formulas, with oriented angles. We omit $\tilde{A}\cdot B$ and $\tilde{A}\bullet B$, which can be obtained from the contractions via \eqref{eq:dots contractions}.

%	Let $ A\in\bigwedge^p X$ and $ B\in\bigwedge^q X$ be nonzero blades, $(e_1,\ldots,e_p)$ and $(f_1,\ldots,f_q)$ be associated principal bases of $[A]$ and $[B]$, with orthoprincipal vectors $e_{d+1}^\perp,\ldots,e_m^\perp$ and $f_{d+1}^\perp,\ldots,f_m^\perp$, where $d = \dim([A]\cap [B])$ and $m=\min\{p,q\}$, and let\footnote{With the understanding that $A_\perp =1$ if $q\geq p$, $B_\perp =1$ if $p\geq q$, $E^\perp = F^\perp = 1$ if $d=m$.}
%	\begin{align*}
%		E\ &= e_1 e_2\cdots e_p, 
%		& E^\perp &= e_{d+1}^\perp e_{d+2}^\perp\cdots e_m^\perp, 
%		& A_\perp &= e_{q+1} e_{q+2}\cdots e_p, \\
%		F\ &= f_1 f_2\cdots f_q, 		 
%		& F^\perp &= f_{d+1}^\perp f_{d+2}^\perp\cdots f_m^\perp, 
%		& B_\perp &= f_{p+1} f_{p+2}\cdots f_q, 		
%	\end{align*}
%	and $J = E^\perp\wedge F$ if $p\leq q$, otherwise $J=E\wedge F^\perp$. Then
	
\begin{proposition}\label{pr:detailed subproducts}
	Let $ A\in\bigwedge^p X$ and $ B\in\bigwedge^q X$ be nonzero blades, $(e_1,\ldots,e_p)$ and $(f_1,\ldots,f_q)$ be associated principal bases of $[A]$ and $[B]$, with orthoprincipal vectors $e_{d+1}^\perp,\ldots,e_m^\perp$ and $f_{d+1}^\perp,\ldots,f_m^\perp$, where $d = \dim([A]\cap [B])$ and $m=\min\{p,q\}$, and $A_\perp$ and $B_\perp$ be the corresponding orthogonal subblades. If $d\neq 0$ let $J=0$, otherwise let $J = e_1^\perp e_2^\perp\cdots e_p^\perp f_1 f_2\cdots f_q$ if $p\leq q$, and $J = e_1 e_2\cdots e_p f_1^\perp f_2^\perp\cdots f_q^\perp$ if $p\geq q$. Then
	\begin{subequations}
		\begin{align}
			\tilde{A}*B &= \|A\|\|B\|\cos\hat{\Theta}_{A,B}, \label{eq:star detailed}\\
			\tilde{A}\glcontr B &= \|A\|\|B\|\cos\Theta_{A,B} \, B_\perp, \label{eq:lcontr detailed}\\
			\tilde{A}\grcontr B &= \|A\|\|B\|\cos\Theta_{B,A} \, \widetilde{A_\perp}, \label{eq:rcontr detailed}\\
			A\wedge B &= \|A\|\|B\|\cos\Theta_{A,B}^\perp\,J. \label{eq:outer detailed}
		\end{align}
	\end{subequations}
	%	\begin{enumerate}[i)]
	%	\item $\tilde{A}*B = 
	%	%\epsilon_{A,B} \|A\|\|B\| \cos\Theta_{A,B} = 
	%	\|A\|\|B\|\cos\hat{\Theta}_{A,B}$. \label{it:A*B Theta}
	%	
	%	\item $\tilde{A}\glcontr B = \|A\|\|B\|\cos\Theta_{A,B} \, B_\perp$. \label{it:left contraction reversion}
	%	
	%	\item $\tilde{A}\grcontr B = \|A\|\|B\|\cos\Theta_{B,A} \, \widetilde{A_\perp}$. \label{it:right contraction reversion}
	%	\item $A\wedge B = \|A\|\|B\|\cos\Theta_{A,B}^\perp\,J$. \label{it:outer}
	%\end{enumerate}
\end{proposition}
\begin{proof}
	\eqref{eq:star detailed} Follows from \eqref{eq:star prod Theta}, \Cref{df:oriented symm complem} and the relation between $\epsilon_{A,B}$ and $\tilde{A}*B$.
	\eqref{eq:lcontr detailed} $\tilde{A}\glcontr B = (\tilde{A}*B_P) B_\perp$, and $\hat{\Theta}_{A,B_P} = \Theta_{A,B_P} = \Theta_{A,B}$, by \Cref{pr:Thetas different dim}.
	\eqref{eq:rcontr detailed} Similar.
	\SELF{$\tilde{A}\grcontr B = \tilde{A}_\perp (\tilde{A}_P*B) 
		= \epsilon_{A,B} \|A\|\|B\| \cos\Theta_{B,A} \tilde{E}_\perp$}
	\eqref{eq:outer detailed} Assume $d = 0$, otherwise both sides vanish. For $p\leq q$, \eqref{eq:AB Theta} gives $A\wedge B = \epsilon_{\tilde{A},B} \|A\|\|B\| \cos\Theta_{[A],[B]}^\perp \,I_1\cdots I_p\,B_\perp$. The reordering of $e_i^\perp$'s from $I_i=e_i^\perp f_i$ cancels the reversion in $\epsilon_{\tilde{A},B}$, and we use \Cref{df:oriented symm complem}. The case $p>q$ is similar.
	\SELF{$A\wedge B = (-1)^{pq} B\wedge A = (-1)^{pq} \epsilon_{B,A} \|B\|\|A\|$ \\ $\cos\Theta_{B,A}^\perp F^\perp \wedge E =$ \\ $\epsilon_{A,B} \|A\|\|B\| \cos\Theta_{A,B}^\perp \, E \wedge F^\perp$}
\end{proof}

Note that both definitions of $J$ give the same result when $p=q$. 
If $d=0$ then $[J]=[A]\oplus[B]$. For $d\neq 0$ we defined $J=0$ for simplicity, but, since $\Theta_{A,B}^\perp = \frac\pi2$ anyway (by \Cref{pr:complementary simple cases}\,\emph{\ref{it:Theta perp pi 2}}), one might as well use some $J$ for which $[J]=[A]+[B]$.
If $p\neq q$ then $\Theta_{A,B}$ and $\Theta_{B,A}$ depend on the choice of principal bases, but in a way that offsets changes in $B_\perp$ and $A_\perp$.

\subsubsection{Products and Reversions}\label{sc:Products with reversions}

As noted after \Cref{pr:geom prod angle bivector}, the geometric product $AB$ involves the relative orientation of $\tilde{A}$ and $B$. The same holds for its component subproducts, and this makes interpreting the resulting orientations less immediate, even unfeasible if $\grade(A)$ is unknown.
To compensate and obtain more geometrically meaningful orientations, these products are often combined with a reversion. 
For example, the Grassmann algebra inner product is $\inner{A,B} = \tilde{A}*B$, the norm is $\|A\|=(\tilde{A}*A)^{\frac12}$, and a slightly different contraction given by $A\lcontr B = \tilde{A}\glcontr B$ (used in \cite{Rosen2019} and in \Cref{sc:Grassmann products}) has its orientation directly related to those of $A$ and $B$ (\Cref{pr:contr orient}). 
This is another geometric price to be paid for algebraic simplicity. 
A product given by $A\diamond B = \tilde{A}B$ would involve $\epsilon_{A,B}$ directly, but would not be associative, making inverses less useful.

The outer product is the only subproduct inheriting the associativity of $AB$, and the only one whose orientation with a reversion ($\tilde{A}\wedge B$) seems less natural.
The proof of \eqref{eq:outer detailed} shows why $A\wedge B$ reflects the orientations of $A$ and $B$ directly: the reversion in $\epsilon_{\tilde{A},B}$ disappears as we reorder the $e_i^\perp$'s to match the usual convention of having first the vectors of $A$ (or their components orthogonal to $B$) then those of $B$, ordered according to the orientation of each blade.

But this begs the question of the reason for such convention.
Of course, it is appropriate that the orientations of the geometric object formed by joining $A$ and $B$ and of the algebraic element $A\wedge B$ used to represent it match each other. But, geometrically, orienting by $A\vartriangle B = \tilde{A}\wedge B$ (first the vectors of $A$ in reverse order, then those of $B$) would have been fine.
Again, the answer lies in algebraic simplicity: $A\vartriangle B$ would have been a worse product, having neither associativity nor alternativity.

\subsection{Commutator and Anticommutator}\label{sc:Commutator and anticommutator}

Let $M,N\in\bigwedge X$.
The \emph{commutator} $M\times N = (MN-NM)/2$ is mainly used with $M$ or $N$ being a bivector, as in this case it preserves grades, and $(\bigwedge^2 X,\times)$ is a Lie algebra \cite{Hestenes1984clifford}. 
But the whole $(\bigwedge X,\times)$ is also a Lie algebra, 
\SELF{https://math.stackexchange.com/questions/841028/what-good-is-the-commutator-product}
and, as we show, the commutator of blades has nice properties for all grades, suggesting that this product may be more interesting than usually recognized.

The \emph{anticommutator} $M\acom N=(MN+NM)/2$ turns the Clifford algebra into a special Jordan algebra \cite{Mccrimmon2003}. It does not seem to have attracted the attention of geometric algebraists, which is strange, given that Clifford algebras are often built from generators satisfying some anticommutation relation (e.g.,  for the Dirac algebra \cite{Hestenes1966} we have $\gamma_\mu \gamma_\nu + \gamma_\nu \gamma_\mu = 2\eta_{\mu\nu}$, which is reduced to $\gamma_\mu \acom \gamma_\nu = \eta_{\mu\nu}$).

These products are linked by the following identities.

\begin{proposition}\label{pr:propriedades acom com}
	Let $M,N\in\bigwedge X$.
	\begin{enumerate}[i)]
		\item $M\acom N + M\times N = MN$. \label{it:acom+com}
		\item $M\acom N - M\times N = NM$. \label{it:acom-com}
		\item $\|M\acom N\|^2 + \|M\times N\|^2 = \frac{\|MN\|^2+\|NM\|^2}{2}$. \label{it:|acom|2+|com|2}
	\end{enumerate}
\end{proposition}
\begin{proof}
%	\emph{(\ref{it:acom+com},\ref{it:acom-com})} Immediate.
	\emph{(\ref{it:|acom|2+|com|2})} % 4\|M\acom N\|^2 + 4\|M\times N\|^2 = 
	$(\widetilde{MN}+\widetilde{NM}) * (MN+NM) + (\widetilde{MN}-\widetilde{NM}) * (MN-NM) = 2(\widetilde{MN})*(MN) + 2(\widetilde{NM})*(NM)$.
\end{proof}

The trivial identity \emph{(\ref{it:acom+com})} becomes more interesting for blades $A$ and $B$, in which case the components of $AB$ are distributed between $A\acom B$ and $A\times B$ according to their grades.
This can be proven algebraically, analyzing the signs of $(A\acom B)^\sim = \tilde{A}\acom \tilde{B} = \pm A\acom B$ and $(A\times B)^\sim = -\tilde{A}\times \tilde{B} = \mp A\times B$.
	\SELF{$= (-1)^{p(q-1)+1} (-1)^{\frac{d(d-1)}{2}} A\times B$, where $d=q-p$}
We give a more geometric argument, which will be useful later.

\begin{proposition}\label{pr:commutator distinct grades}
	Given blades $A\in\bigwedge^p X$ and $B\in\bigwedge^q X$, with $p\leq q$, and a PO decomposition of $B$ w.r.t.\! $A$, we have
	\begin{equation*}
		A\times B = 
		\begin{cases}
			(A \times B_P) \wedge B_\perp = \sum\limits_{k=0}^\infty \inner{AB}_{q-p+4k+2} \ \text{ if }\, p(q-1) \text{ is even,} \\
			(A \acom B_P) \wedge B_\perp = \sum\limits_{k=0}^\infty \inner{AB}_{q-p+4k} \hspace{10pt} \ \text{ if }\, p(q-1) \text{ is odd,}
		\end{cases}
	\end{equation*}
	and likewise for the anticommutator, with the conditions swapped.
\end{proposition}
\begin{proof}
	As seen in \Cref{sc:geometric product}, if $p=q$ the only difference between $\tilde{A}B$ and $\tilde{B}A$ is that components with an odd number of $I_i$'s switch signs. Therefore $A\times B = (-1)^{\frac{p(p-1)}{2}} (\tilde{A}B-\tilde{B}A)/2$ has the components of grades $4k+2$ ($k\in\N$) of $AB$, while $A\acom B$ has those of grades $4k$.
	And when $p<q$ we have $2A\times B = AB_PB_\perp - B_PB_\perp A = \left(AB_P -(-1)^{p(q-p)} B_PA\right) B_\perp$, and likewise for the anticommutator.
\end{proof}

\begin{example}
	In $\bigwedge\R^2$, for $M=i$ and $N=1+ij$ we have $M\acom N = i$ and $M\times N = j$, which decompose $MN=i+j$, but not by grades. This shows it is not enough that one of the multivectors be a blade.
\end{example}

%For blades we have more identities.

\begin{proposition}\label{pr:propriedades acom com blades}
	Let $A,B\in\bigwedge X$ be blades. 
	\begin{enumerate}[i)]
		\item $(A\acom B)\times(A\times B)=0$.\label{it:acom com com}
		\item $(A\acom B)*(A\times B)=0$.\label{it:acom * com}
		\item $(A\acom B)^2 - (A\times B)^2 = A^2B^2$.\label{it:acom2-com2}
			\SELF{Cuidado, $A^2B^2 \neq (AB)^2$}
		\item For unit blades of same grade, $(A\acom B)^2 - (A\times B)^2 = 1$.\label{it:acom2-com2 1}
		\item $\|A\acom B\|^2 + \|A\times B\|^2 = \|A\|^2 \|B\|^2$. \label{it:|acom|2+|com|2 blades} 
	\end{enumerate}	 
\end{proposition}
\begin{proof}
	\emph{(\ref{it:acom com com})} $(AB+BA)(AB-BA) - (AB-BA)(AB+BA) = 2 BAAB - 2 ABBA = 0$, since for blades $A^2$ and $B^2$ are scalars.
	\emph{(\ref{it:acom * com})} By \Cref{pr:commutator distinct grades}, $A\acom B$ and $A\times B$ have no common grades.
	\emph{(\ref{it:acom2-com2})} $(AB+BA)^2 - (AB-BA)^2 = 2ABBA + 2BAAB = 4A^2B^2$.
	\emph{(\ref{it:acom2-com2 1})} For unit $p$-blades, $A^2=B^2=(-1)^\frac{p(p-1)}{2}$.
	\emph{(\ref{it:|acom|2+|com|2 blades})} Follows from \Cref{pr:propriedades acom com}\emph{\ref{it:|acom|2+|com|2}}.
\end{proof}

\begin{example}
	Let $M=N=1+i\in\bigwedge \R^2$. Then $MN=NM=M\acom N=2+2i$ and $M\times N=0$. One can check that, while \Cref{pr:propriedades acom com}\emph{\ref{it:|acom|2+|com|2}} holds, \Cref{pr:propriedades acom com blades}\emph{\ref{it:|acom|2+|com|2 blades}} does not (as these are not blades).
\end{example}

These products are related to hyperbolic functions (see \Cref{sc:Hyperbolic Functions}) of the angle bivector.

\begin{theorem}\label{pr:commutators hyperbolic}
	For blades $A,B\in\bigwedge^p X$ we have $\tilde{A}\acom B = \|A\|\|B\|\cosh \btheta_{A,B}$ and $\tilde{A}\times B = \|A\|\|B\|\sinh \btheta_{A,B}$. 
\end{theorem}
\begin{proof}
	As grades are equal, \Cref{pr:geom prod angle bivector} gives $\tilde{A}\acom B = (\tilde{A}B+\tilde{B}A)/2 = \|A\|\|B\| (e^{\btheta_{A,B}} + e^{-\btheta_{A,B}})/2$, and likewise for $\tilde{A}\times B$.
\end{proof}

These formulas may give the impression that these products are not submultiplicative for blades, which is false, by \Cref{pr:propriedades acom com}\emph{\ref{it:|acom|2+|com|2 blades}}. What happens is that, as the $I_i$'s in $\btheta_{A,B}$ square to $-1$, these hyperbolic functions expand in terms of trigonometric functions of the $\theta_i$'s.
Indeed, for $V=[A]$ and $W=[B]$ we have $\cosh \btheta_{A,B} = \epsilon_{A,B} \cosh\btheta_{V,W}$ (likewise for $\sinh$), and $\cosh\btheta_{V,W}$ (resp. $\sinh$) is given by the terms in \Cref{pr:exp expansion theta_i} with an even (resp. odd) number of $I_i$'s, so that $\|\cosh \btheta_{A,B}\|^2 + \|\sinh \btheta_{A,B}\|^2 = \|e^{\btheta_{A,B}}\|^2 = 1$ (as in \Cref{pr:properties cosh sinh homog}\emph{\ref{it:r23}}).

\begin{example}
	Let $c_i=\cos\theta_i$, $s_i=\sin\theta_i$, and $p=\dim V=\dim W$.
	If $p=2$ then $\cosh \btheta_{V,W} = c_1c_2 + s_1s_2I_1I_2$ and $\sinh \btheta_{V,W} = s_1c_2I_1 + c_1s_2I_2$.
	If $p=3$ then $\cosh \btheta_{V,W} = c_1c_2c_3 + s_1s_2c_3I_1I_2 + s_1c_2s_3I_1I_3 + c_1s_2s_3I_2I_3$ and $\sinh \btheta_{V,W} = s_1c_2c_3I_1 + c_1s_2c_3I_2+c_1c_2s_3I_3 + s_1s_2s_3I_1I_2I_3$.
\end{example}

With \Cref{pr:commutators hyperbolic}, we can see that many properties presented here reflect others of hyperbolic functions given in \Cref{sc:Hyperbolic Functions}.
For example, \Cref{pr:propriedades acom com blades}\emph{\ref{it:acom2-com2 1}} corresponds directly to \ref{pr:properties cosh sinh M}\emph{\ref{it:cosh2-sinh2}}.
If grades are distinct, the orthogonal subblade can demand some extra effort from us to see the correspondence, but it exists. 
For example, consider \Cref{pr:propriedades acom com blades}\emph{\ref{it:acom2-com2}} with $\grade(A)=p \leq q=\grade(B)$. Using a PO decomposition and \Cref{pr:commutator distinct grades}, and since $B_\perp$ commutes with the even multivectors $A\acom B_P$ and $A\times B_P$, we have
\begin{align*}
	(A\acom B)^2 - (A\times B)^2 
	&= (-1)^{(q-1)p} \bigl[ (A\acom B_P)^2 - (A\times B_P)^2 \bigr] B_\perp^2 \\
	&= (-1)^{(q-p)p} \|A\|^2\|B_P\|^2 (\cosh^2 \btheta_{A,B} - \sinh^2 \btheta_{A,B}) B_\perp^2 \\
	&= (-1)^{(q-p)p} A^2 B_P^2 B_\perp^2 \\
	&= A^2 (B_P B_\perp)^2 = A^2B^2. 
\end{align*}

%\begin{proof}[Proof of \Cref{pr:propriedades acom com blades}\,\ref{it:acom2-com2}]
%	Let $A\in\bigwedge^p X$ and $B\in\bigwedge^q X$. If $p=q$ then $(A\acom B)^2 - (A\times B)^2 = \|A\|^2\|B\|^2 (\cosh^2 \btheta_{A,B} - \sinh^2 \btheta_{A,B}) = \|A\|^2\|B\|^2 = A^2B^2$.
%	Otherwise we can assume $p<q$ and, as $B_\perp$ commutes with $A\acom B_P$ and $A\times B_P$,
%	\SELF{$A\acom B_P$ and $A\times B_P$ are even, and $B_\perp$ is completely orthogonal}
%	\Cref{pr:commutator distinct grades} gives 
%	\begin{align*}
%		(A\acom B)^2 - (A\times B)^2 
%		&= (-1)^{(q-1)p} \bigl[ (A\acom B_P)^2 - (A\times B_P)^2 \bigr] B_\perp^2 \\
%		%		&= (-1)^{(q-p)p} \|A\|^2\|B_P\|^2 B_\perp^2 \\
%		&= (-1)^{(q-p)p} A^2 B_P^2 B_\perp^2 \\
%		&= A^2 (B_P B_\perp)^2. \qedhere
%	\end{align*}
%\end{proof}

\begin{proposition}\label{pr:hyperbolic=0}
	Given blades $A,B\in\bigwedge^p X$, let $r$ be the number of principal angles that are $\frac\pi2$. Then $\cosh \btheta_{A,B} = 0$ (resp. $\sinh$) if, and only if, $r$ is odd (resp. even) and any other principal angle is $0$.
\end{proposition}
\begin{proof}
	If $\theta_1\leq\cdots\leq\theta_p$ are the principal angles then $c_i = \cos\theta_i \neq 0 $ for $i\leq p-r$, and $s_i = \sin\theta_i \neq 0$ for $i>p-r$.
	As $\cosh \btheta_{A,B}$ has components with all possible products of an even number of $s_i$'s, multiplied by $c_i$'s of the other indices, if $r$ is even it has a component with $c_1\cdots c_{p-r}s_{p-r+1}\cdots s_p \neq 0$.
	So $\cosh \btheta_{A,B} = 0$ implies $r$ is odd. But then it will have a component with $c_1\cdots c_{p-r-1}s_{p-r}\cdots s_p$, which must vanish, so $\theta_{p-r}=0$. 
	Conversely, if $r$ is odd and $\theta_{p-r}=0$, all components will have some $c_i=0$ or some $s_i=0$. The proof for $\sinh\btheta_{A,B}$ is similar.
\end{proof}
%\begin{proof}
%	We have $c_i \neq 0 $ for $i\leq p-r$, and $s_i \neq 0$ for $i>p-r$.
%	As $\sinh \btheta_{A,B}$ has components with all possible products of an odd number of $s_i$'s, multiplied by $c_i$'s of the other indices, if $r$ is odd it has a component with $c_1\cdots c_{p-r}s_{p-r+1}\cdots s_p \neq 0$.
%	So $\sinh \btheta_{A,B} = 0$ implies $r$ is even, and there will be a component with $c_1\cdots c_{p-r-1}s_{p-r}\cdots s_p$ which must vanish, so $\theta_{p-r}=0$. 
%	Conversely, if $r$ is even and $\theta_{p-r}=0$, all components will have some $c_i=0$ or some $s_i=0$. The proof for $\cosh\btheta_{A,B}$ is similar.
%\end{proof}

%\begin{corollary}
%	Same grade blades $A$ and $B$ commute (resp. anticommute) if, and only if, $A=C D_1$ and $B=C D_2$ for completely orthogonal blades $C,D_1,D_2$, with $D_1$ and $D_2$ having even (resp. odd) grade.
%\end{corollary}

\begin{proposition}
	Blades $A\in\bigwedge^p X$ and $B\in\bigwedge^q X$, with $p\leq q$, commute (resp. anticommute) if, and only if, $A=M A'$ and $B=M B' B''$ for completely orthogonal blades $M, A', B', B''$, with $\grade(A') = \grade(B')$ having the same (resp. opposite) parity of $p(q-1)$.
\end{proposition}
\begin{proof}
	Follows from \Cref{pr:commutators hyperbolic,pr:hyperbolic=0} if $p=q$ (with $B''=1$), and also \Cref{pr:commutator distinct grades} if $p<q$ (with $B''=B_\perp$).
\end{proof}

\begin{example}
	In \Cref{ex:angle bivector,ex:oriented angle bivector}, the principal angles are $(0,\frac\pi2,\frac\pi2)$, $E\acom F = \cosh \btheta_{E,F} = e_2f_2e_3f_3$ and $E\times F = \sinh \btheta_{E,F} = 0$.
	We also have $(E\acom F)^2 = 1$ and $E^2=F^2=-1$.
\end{example}

\begin{example}
	In \Cref{ex:geom product 1a}, $A\acom B = -(2+3\I_1\I_2)/5\sqrt{2}$  has the terms of $AB = -\tilde{A}B$ that do not depend on the direction of the projections (from $[A]$ to $[B]$ or vice-versa), and $A\times B = -(6\I_1+\I_2)/5\sqrt{2}$ has those that do. 
	We have $(A\acom B)^2 = (13+12I_1I_2)/50$, $(A\times B)^2=(-37+12I_1I_2)/50$ and $A^2 = B^2 = -1$.
	Also, $\|A\acom B\|^2=\frac{13}{50}$, $\|A\times B\|^2=\frac{37}{50}$, and $\|A\|=\|B\|=1$.
	Finally, $(A\acom B)\times(A\times B)=0$, since the $I_i$'s commute, and $(A\acom B)*(A\times B)=0$, as their components have no common grades.
\end{example}

\begin{example}
	In \Cref{ex:geom product 3}, $\cosh\btheta_{A,B} = \frac{1}{2}I_1I_2$ and $\sinh\btheta_{A,B} = \frac{\sqrt{3}}{2}\I_2$, $A\acom B = (-\tilde{A}B+\tilde{B}A)/2 = -\frac{1}{2}I_1 I_2 f_3f_4$ and $A\times B = (-\tilde{A}B-\tilde{B}A)/2 = -\frac{\sqrt{3}}{2}\I_2 f_3f_4$.
	We have $p=2$ and $q=4$, so $p(q-1)$ is even, and $B_\perp = f_3f_4$.
	Note that $A\acom B = (e_1e_2 \acom f_1f_2)B_\perp  = -\cosh\btheta_{A,B} B_\perp$ has grade $6=q-p+4$,
	while $A\times B = (e_1e_2 \times f_1f_2)B_\perp = -\sinh\btheta_{A,B} B_\perp$ has grade $4=q-p+2$.
	Also, they commute, $(A\acom B)^2 = -\frac14$, $(A\times B)^2=\frac34$, $A^2 = -1$ and $B^2=1$. 
	
	With $e_1$ instead of $A$, we have $\btheta_{e_1,B} = \btheta_{e_1,f_1} = \frac\pi6 I_1$, $\cosh \btheta_{e_1,B} = \frac{\sqrt{3}}{2}$ and $\sinh \btheta_{e_1,B} = \frac{1}{2} I_1$. 
	Now $p=1$ and $q=4$, so $p(q-1)$ is odd, and $B_\perp = f_2f_3f_4$. 
	We find that $e_1\acom B = \frac{1}{2} I_1 f_2f_3f_4 = (e_1\times f_1) B_\perp = \sinh \btheta_{e_1,B} B_\perp$ has grade $5=q-p+2$, while $e_1\times B = \frac{\sqrt{3}}{2} f_2f_3f_4 = (e_1\acom f_1)B_\perp = \cosh \btheta_{e_1,B} B_\perp$ has grade $3=q-p$.
	They commute, $(e_1\acom B)^2 = \frac14$ and $(e_1\times B)^2=-\frac34$. 
\end{example}

\section*{Acknowledgments}

The author would like to thank Dr. K.\,Scharnhorst for his comments and for suggesting references, and the anonymous referees who encouraged the conversion of previous versions of the manuscript (and the author himself) to the formalism of geometric algebra.

\vspace{6pt}

\noindent
\textbf{Note.} This article has been posted to the arXiv e-print repository, with the identifier arXiv:1910.07327

%\bibliographystyle{amsplain} 
%\bibliography{../../../../Bibliografia_Linear_Geometry/Linear_Geometry}

\begin{thebibliography}{10}
	
	\bibitem{Afriat1957}
	S.~Afriat, \emph{Orthogonal and oblique projectors and the characteristics of
		pairs of vector spaces}, Math. Proc. Cambridge Philos. Soc. \textbf{53}
	(1957), no.~4, 800--816.
	
	\bibitem{Bjorck1973}
	A.~Bjorck and G.~Golub, \emph{Numerical methods for computing angles between
		linear subspaces}, Math. Comp. \textbf{27} (1973), no.~123, 579.
	
	\bibitem{Deutsch1995}
	F.~Deutsch, \emph{The angle between subspaces of a {H}ilbert space},
	Approximation Theory, Wavelets and Applications, Springer Netherlands, 1995,
	pp.~107--130.
	
	\bibitem{Dorst2001}
	L.~Dorst, \emph{Honing geometric algebra for its use in the computer sciences},
	Geometric Computing with Clifford Algebras (G.~Sommer, ed.), Springer Berlin
	Heidelberg, 2001, pp.~127--152.
	
	\bibitem{Dorst2002}
	\bysame, \emph{The inner products of geometric algebra}, Applications of
	Geometric Algebra in Computer Science and Engineering (L.~Dorst, C.~Doran,
	and J.~Lasenby, eds.), Birkh{\"a}user, 2002, pp.~35--46.
	
	\bibitem{Dorst2007}
	L.~Dorst, D.~Fontijne, and S.~Mann, \emph{Geometric algebra for computer
		science: an object-oriented approach to geometry}, Elsevier, 2007.
	
	\bibitem{Drmac2000}
	Z.~Drmac, \emph{On principal angles between subspaces of {E}uclidean space},
	SIAM J. Matrix Anal. Appl. \textbf{22} (2000), no.~1, 173--194.
	
	\bibitem{Galantai2008}
	A.~Gal{\'{a}}ntai, \emph{Subspaces, angles and pairs of orthogonal
		projections}, Linear Multilinear Algebra \textbf{56} (2008), no.~3, 227--260.
	
	\bibitem{Galantai2006}
	A.~Gal{\'{a}}ntai and C.~J. Heged{\H{u}}s, \emph{Jordan's principal angles in
		complex vector spaces}, Numer. Linear Algebra Appl. \textbf{13} (2006),
	no.~7, 589--598.
	
	\bibitem{Gluck1967}
	H.~Gluck, \emph{Higher curvatures of curves in {E}uclidean space, {II}}, Amer.
	Math. Monthly \textbf{74} (1967), no.~9, 1049--1056.
	
	\bibitem{Golub2013}
	G.~H. Golub and C.~F. Van~Loan, \emph{Matrix computations}, Johns Hopkins
	University Press, 2013.
	
	\bibitem{Gull1993imaginary}
	S.~Gull, A.~Lasenby, and C.~Doran, \emph{Imaginary numbers are not real -- the
		geometric algebra of spacetime}, Found. Phys. \textbf{23} (1993), no.~9,
	1175--1201.
	
	\bibitem{Gunawan2005}
	H.~Gunawan, O.~Neswan, and W.~Setya-Budhi, \emph{A formula for angles between
		subspaces of inner product spaces}, Beitr. Algebra Geom. \textbf{46} (2005),
	no.~2, 311--320.
	
	\bibitem{Hawidi1966}
	H.~M. Hawidi, \emph{Vzaimnoe raspolozhenie podprostranstv v konechnomernom
		unitarnom prostranstve [{T}he relative position of subspaces in a
		finite-dimensional unitary space]}, Ukra{\"i}n. Mat. Zh. \textbf{18} (1966),
	no.~6, 130--134.
	
	\bibitem{Hestenes1966}
	D.~Hestenes, \emph{Space-time algebra}, Gordon and Breach Science Publishers,
	New York, 1966.
	
	\bibitem{Hestenes1999}
	\bysame, \emph{New foundations for classical mechanics}, 2 ed., Springer
	Netherlands, 1999.
	
	\bibitem{Hestenes2005}
	\bysame, \emph{Gauge theory gravity with geometric calculus}, Found. Phys.
	\textbf{35} (2005), no.~6, 903--970.
	
	\bibitem{Hestenes1984clifford}
	D.~Hestenes and G.~Sobczyk, \emph{Clifford algebra to geometric calculus}, D.
	Reidel, Dordrecht, 1984.
	
	\bibitem{Hitzer2010a}
	E.~Hitzer, \emph{Angles between subspaces computed in {C}lifford algebra}, AIP
	Conference Proceedings, vol. 1281, AIP, 2010, pp.~1476--1479.
	
	\bibitem{Hitzer2012}
	\bysame, \emph{{Introduction to {C}lifford's geometric algebra}}, Journal of
	SICE \textbf{51} (2012), no.~4, 338--350.
	
	\bibitem{Hotelling1936}
	H.~Hotelling, \emph{Relations between two sets of variates}, Biometrika
	\textbf{28} (1936), no.~3/4, 321--377.
	
	\bibitem{Jiang1996}
	S.~Jiang, \emph{Angles between {E}uclidean subspaces}, Geom. Dedicata
	\textbf{63} (1996), 113--121.
	
	\bibitem{Jordan1875}
	C.~Jordan, \emph{Essai sur la g{\'e}om{\'e}trie {\`a} n dimensions}, Bull. Soc.
	Math. France \textbf{3} (1875), 103--174.
	
	\bibitem{Kozlov2000}
	S.~E. Kozlov, \emph{Geometry of real {G}rassmann manifolds. {P}art {III}}, J.
	Math. Sci. \textbf{100} (2000), no.~3, 2254--2268.
	
	\bibitem{Lounesto1993}
	P.~Lounesto, \emph{Marcel {R}iesz's work on {C}lifford algebras}, Clifford
	Numbers and Spinors (E.~F. Bolinder and P.~Lounesto, eds.), Springer
	Netherlands, 1993, pp.~215--241.
	
	\bibitem{Lounesto2001}
	\bysame, \emph{Clifford algebras and spinors}, 2nd ed., vol. 286, Cambridge
	University Press, 2001.
	
	\bibitem{Macdonald2017}
	A.~Macdonald, \emph{A survey of geometric algebra and geometric calculus}, Adv.
	Appl. Clifford Al. \textbf{27} (2017), 853--891.
	
	\bibitem{Mandolesi_Grassmann}
	A.~L.~G. Mandolesi, \emph{Grassmann angles between real or complex subspaces},
	arXiv:1910.00147 (2019).
	
	\bibitem{Mandolesi_Pythagorean}
	\bysame, \emph{Projection factors and generalized real and complex
		{P}yth\-a\-go\-re\-an theorems}, Adv. Appl. Clifford Al. \textbf{30} (2020),
	no.~43.
	
	\bibitem{Mccrimmon2003}
	K.~Mc{C}rimmon, \emph{A taste of {J}ordan algebras}, Springer-Verlag New York,
	2004.
	
	\bibitem{Miao1992}
	J.~Miao and A.~Ben-Israel, \emph{On principal angles between subspaces in
		${R}^n$}, Linear Algebra Appl. \textbf{171} (1992), 81--98.
	
	\bibitem{Rosen2019}
	A.~Ros{\'e}n, \emph{Geometric multivector analysis}, Springer-Verlag, 2019.
	
	\bibitem{Scharnhorst2001}
	K.~Scharnhorst, \emph{Angles in complex vector spaces}, Acta Appl. Math.
	\textbf{69} (2001), no.~1, 95--103.
	
	\bibitem{Stolfi1991}
	J.~Stolfi, \emph{Oriented projective geometry: A framework for geometric
		computations}, Academic Press, 1991.
	
	\bibitem{Wedin1983}
	P.~Wedin, \emph{On angles between subspaces of a finite dimensional inner
		product space}, Matrix Pencils (B.~K{\aa}gstr{\"o}m and A.~Ruhe, eds.),
	Lecture Notes in Mathematics, vol. 973, Springer, 1983, pp.~263--285.
	
	\bibitem{Ye2016}
	K.~Ye and L.H. Lim, \emph{Schubert varieties and distances between subspaces of
		different dimensions}, SIAM J. Matrix Anal. Appl. \textbf{37} (2016), no.~3,
	1176--1197.
	
\end{thebibliography}

\providecommand{\bysame}{\leavevmode\hbox to3em{\hrulefill}\thinspace}
\providecommand{\MR}{\relax\ifhmode\unskip\space\fi MR }
\providecommand{\MRhref}[2]{%
	\href{http://www.ams.org/mathscinet-getitem?mr=#1}{#2}
}
\providecommand{\href}[2]{#2}

\appendix

\section{Blade products in Grassmann algebra}\label{sc:Grassmann products}

We present some results of \Cref{sc:Subproducts and Grassmann angles} in terms of Grassmann algebra, to make them more accessible to researchers who might not be familiarized with Clifford algebra.
We provide direct proofs requiring (mostly) only \Cref{sc:preliminaries}.

Here $X$ can be a real or complex vector space, with inner product  $\inner{\cdot,\cdot}$ (Hermitian product in the complex case, with conjugate-linearity in the first argument).
The complex case is important for applications, but requires some adjustments in \Cref{sc:preliminaries}, which are indicated in footnotes.

%\subsection{Inner product}\label{sc:inner product}

Formulas relating the inner product of same grade blades to some angle are well known in the real case \cite{Gluck1967,Jiang1996}.
The following one also works in complex spaces and for distinct grades.

\begin{theorem}\label{pr:Grassmann angle blade product}
	$\inner{A, B} =  \| A\|\| B\|\cos\hat{\Theta}_{A,B}$ for blades $A,B\in\bigwedge X$.
\end{theorem}
\begin{proof}
	If grades are equal, follows from \Cref{pr:A*B theta_i} (as $\inner{A,B}=\tilde{A}*B$), \eqref{eq:Theta prod cos} and \Cref{df:oriented symm complem}. 
	Otherwise both sides vanish.
\end{proof}

%\subsection{Exterior product}\label{sc:Exterior product}

The exterior product satisfies a similar formula.

\begin{theorem}\label{pr:norm wedge product}
	$\| A\wedge B\|=\| A\|\| B\| \cos\Theta^\perp_{[A],[B]}$ for blades $A,B\in\bigwedge X$. 
\end{theorem}
\begin{proof}
	Let $\P^\perp=\Proj_{B^\perp}$. Then $\| A\wedge B\| = \|(\P^\perp A)\wedge B\| = \|\P^\perp A\| \| B\|$, and the result follows from \Cref{pr:properties Grassmann}\emph{\ref{it:Theta Proj}}.
\end{proof}

This result corresponds to \eqref{eq:outer prod Theta perp}, and our comments about that formula (after \Cref{pr:contraction pperp}) also apply here.

%\subsection{Contraction or interior product}\label{sc:Interior Product}

Contraction or interior product by a vector is widely used in Geometry and Physics, but as its generalization for multivectors is less known, we give a brief description here.
The following contraction is related to that of \Cref{sc:Subproducts and Grassmann angles} by $A\lcontr B = \tilde{A}\glcontr B$.
For more details, see \cite{Dorst2007,Rosen2019}.

\begin{definition}
	The \emph{(left) contraction}
	\SELF{Our results adapt for a right one, given by $ B\,\rcontr A = (-1)^{p(q-p)}  A\lcontr B$.}
	$ A\lcontr B$ of $A\in\bigwedge^p X$ on $B \in\bigwedge^q X$ is the unique element of $\bigwedge^{q-p} X$ such that $\inner{ C, A \lcontr   B} = \inner{ A\wedge C, B}$ for all $C\in\bigwedge^{q-p} X$.
\end{definition}

If $p=q$ then $ A \lcontr   B = \inner{ A, B}$, so the contraction generalizes the inner product for distinct grades, but giving a $(q-p)$-vector instead of a scalar. 
\SELF{if $p=0$ then $ A \lcontr   B = \bar{ A}\,  B$}
For $p\neq q$ this product is asymmetric (in general, $ A\lcontr B \neq  B\lcontr A$), with $ A\lcontr B =  0$ if $p>q$.
In the complex case it is conjugate-linear in $ A$ and linear in $ B$.

Let $A\in\bigwedge^p X$ and $B \in\bigwedge^q X$ be nonzero blades, and $B = B_P\wedge B_\perp$ be a PO decomposition\footnote{See \Cref{sc:Subspaces of different dimensions}.} 
of $ B$ w.r.t.\! $A$.

\begin{proposition}\label{pr:contraction projective orthogonal}
	$A \lcontr B = \inner{A, B_P} \, B_\perp$.
\end{proposition}
\begin{proof}
	As $B_\perp$ is completely orthogonal to $A$, for any $C\in\bigwedge^{q-p} X$ we have $\inner{ A\wedge C, B_P\wedge B_\perp} = \inner{ A, B_P} \inner{C, B_\perp} = \inner{C, \inner{ A, B_P} B_\perp}$.
\end{proof}

So $ A\lcontr B$ performs an inner product of $ A$ with a subblade of $ B$ where it projects, leaving another subblade of $B$ completely orthogonal to $A$.

\begin{corollary}
	$B_\perp =  B_P\lcontr B/\|B\|^2$.
	\SELF{$=\frac{P A}{\|P A\|}\lcontr B$, and $ A=P A+ A'$ with $ A'\pperp B$}
\end{corollary}

%The following formula is useful for calculations.

%\begin{proposition}\label{pr:formula interior}
%	For $ A\in\bigwedge^p X$ and any blade $ B\in\bigwedge^q X$, with $p\leq q$, 
%	\SELF{Inclui $p=0$}
%	\SELF{Geometric Algebra, Chisolm, eq. 95, dá uma fórmula semelhante pra $\glcontr$}
%	\begin{equation}\label{eq:contraction coordinate decomposition} 
%		 A \lcontr   B = \sum_{\ii\in\II_p^q} \epsilon_\ii \,\inner{ A, B_\ii}\,  B_{\ii'},
%	\end{equation}
%	where $\epsilon_\ii$, $ B_\ii$ and $ B_{\ii'}$ are as in \eqref{eq:multiindex decomposition} for any decomposition $ B = w_1\wedge\ldots\wedge w_q$.
%\end{proposition}
%\begin{proof}
%	Follows from \eqref{eq:inner wedge} and \eqref{eq: interior exterior adjoints}.
%\end{proof}
%
%
%\begin{example}
%	For $v\in X$, $ A\in\bigwedge^2 X$ and $ B=w_1\wedge w_2\wedge w_3\in\bigwedge^3 X$,
%	\begin{align*}
%		v\lcontr B &= \inner{v,w_1}w_2\wedge w_3 - \inner{v,w_2}w_1\wedge w_3 + \inner{v,w_3}w_1\wedge w_2, \\
%		 A\lcontr  B &= \inner{ A,w_1\wedge w_2}w_3  - \inner{ A,w_1\wedge w_3}w_2 + \inner{ A,w_2\wedge w_3}w_1.
%	\end{align*}
%\end{example}

\begin{theorem}\label{pr:contraction}
	$A \lcontr   B =  \| A\|\| B\|  \cos\Theta_{A,B}\, B_\perp$.
\end{theorem}
\begin{proof}
	Follows from \Cref{pr:Grassmann angle blade product} and \Cref{pr:contraction projective orthogonal,pr:Thetas different dim} if $p\leq q$, otherwise $\Theta_{A,B}=\frac \pi 2$ and both sides vanish. 
\end{proof}

%\begin{corollary} 
%	$\| A\lcontr B\| = \| A\|\| B\| \cos \Theta_{[A],[B]}=\|P_{B} A\|\| B\|$.
%\end{corollary}

\begin{corollary}\label{pr:contr orient}
	$A \lcontr B = \epsilon_{A,B} \|P_{B} A\| \|B\| \, B_\perp$.
\end{corollary}

\begin{corollary}
	$A\lcontr B=0 \Leftrightarrow A\pperp B$.
\end{corollary}

%These results generalize \Cref{pr:Theta inner blades,pr:null inner part ort blades}, and give another characterization for the Grassmann angle, as $\Theta_{V,W}=\arccos\frac{\| A\lcontr B\|}{\| A\|\| B\|}$ if $V=[ A]$ and $W=[ B]$. 

\section{Hyperbolic Functions of Multivectors}\label{sc:Hyperbolic Functions}

Hyperbolic functions of multivectors are defined as usual, in terms of exponentials or power series \cite{Hestenes1999}.

\begin{definition}
	For any $M\in\bigwedge X$, $\cosh M = \frac{e^M + e^{-M}}{2} = \sum\limits_{k=0}^\oo \frac{M^{2k}}{(2k)!}$ and $\sinh M = \frac{e^M - e^{-M}}{2}  = \sum\limits_{k=0}^\oo \frac{M^{2k+1}}{(2k+1)!}$.
\end{definition}

Though we only need the bivector case for \Cref{sc:Commutator and anticommutator}, we present properties of these funcions in more generality, as the literature on them is rather scarce (a few other properties, mostly for blades, can be found in \cite[p.\,77]{Hestenes1999}).

\begin{proposition}\label{pr:properties cosh sinh M}
	Let $M,N\in\bigwedge X$.
	\begin{enumerate}[i)]
		\item $\cosh M + \sinh M = e^M$. \label{it:cosh+sinh}
		\item $\cosh M - \sinh M = e^{-M}$. \label{it:cosh-sinh}
		\item $\cosh(-M) = \cosh M$ and $\sinh(-M) = -\sinh M$. \label{it:cosh sinh -M}
		\item $(\cosh M)^\sim = \cosh \tilde{M}$ and $(\sinh M)^\sim = \sinh \tilde{M}$. \label{it:cosh sinh reverse}
		\item If $M\times N=0$ then $\cosh(M)\times\sinh(N) = 0$. \label{it:coshxsinh}
		\item $\cosh^2 M - \sinh^2 M = 1$. \label{it:cosh2-sinh2}
	\end{enumerate}
\end{proposition}
\begin{proof}
	\emph{(\ref{it:cosh+sinh} -- \ref{it:cosh sinh reverse})} Immediate.
	\emph{(\ref{it:coshxsinh})} If $M\times N=0$ then $e^M e^N = e^{M+N}$, and so $\cosh M \sinh N = \left(e^{M+N}-e^{M-N}+e^{N-M}-e^{-M-N}\right)/4 = \sinh N \cosh M$. 
	\emph{(\ref{it:cosh2-sinh2})} $\cosh^2 M = (e^{2M}+2+e^{-2M})/4$ and $\sinh^2 M = (e^{2M}-2+e^{-2M})/4$.
\end{proof}

\begin{proposition}\label{pr:properties cosh sinh homog}
	Let $H\in\bigwedge^p X$ be homogeneous of grade $p$, and $r=p\!\!\mod 4$.\SELF{$M=i+jkl$, $\tilde{M}\neq\pm M$ e $\tilde{M}\times M\neq 0$}
	\begin{enumerate}[i)]
		\item $(\cosh H)^\sim = \cosh H$ and $(\sinh H)^\sim = (-1)^{\frac{p(p-1)}{2}}\sinh H$. \label{it:symmetry cosh sinh}
		\item $\cosh H \in \bigoplus\limits_{k\in\N} \bigwedge^{4k}X$ and $\sinh H \in \bigoplus\limits_{k\in\N} \bigwedge^{4k+r}X$. \label{it:grades cosh sinh}
		\item If $r\neq 0$ then $\cosh H * \sinh H = 0$. \label{it:cosh * sinh}	
		\item If $r = 0$ then \label{it:r0}
		$\begin{cases}
			\|\cosh H\|^2 + \|\sinh H\|^2 = \frac{\|e^H\|^2 + \|e^{-H}\|^2}{2}, \\
			\|\cosh H\|^2 - \|\sinh H\|^2 = 1, \\
			\|\cosh H\| \geq 1.	
		\end{cases}$ 
		\item If $r = 1$ then  \label{it:r1}
		$\begin{cases}
			\|\cosh H\|^2 + \|\sinh H\|^2 = \|e^H\|^2, \\
			\|\cosh H\|^2 - \|\sinh H\|^2 = 1, \\
			\|\cosh H\| \geq 1 \text{ and } \|e^H\| = \|e^{-H}\| \geq 1.		
		\end{cases}$
		\item If $r = 2$ or $3$ then \label{it:r23} 
		$\begin{cases}
			\|\cosh H\|^2 + \|\sinh H\|^2 = 1, \\
			\|\cosh H\|^2 - \|\sinh H\|^2 = \inner{e^H}_0^2, \\
			\|\cosh H\| \leq 1, \|\sinh H\| \leq 1 \text{ and } \|e^H\| = 1.
		\end{cases}$
	\end{enumerate}
\end{proposition}
\begin{proof}
	\emph{(\ref{it:symmetry cosh sinh})} Follows from \Cref{pr:properties cosh sinh M}\,\emph{\ref{it:cosh sinh -M}} and \emph{\ref{it:cosh sinh reverse}}, as $\tilde{H}=(-1)^{\frac{p(p-1)}{2}} H$.
	\emph{(\ref{it:grades cosh sinh})} Follows from \emph{\ref{it:symmetry cosh sinh}}, as components of $\cosh H$ are even, and those of $\sinh H$ have the parity of $p$.
	\emph{(\ref{it:cosh * sinh})} By \emph{\ref{it:grades cosh sinh}}, these functions have no components of same grade when $r\neq 0$.
	\emph{(\ref{it:r0})} $\tilde{H}=H$, so $4\|\cosh H\|^2 = (e^{\tilde{H}} + e^{-\tilde{H}}) * (e^H + e^{-H}) = \|e^H\|^2 + 2 + \|e^{-H}\|^2$ and $4\|\sinh H\|^2 = \|e^H\|^2 - 2 + \|e^{-H}\|^2$.
	\emph{(\ref{it:r1})} Likewise, but by \emph{\ref{it:grades cosh sinh}} and \Cref{pr:properties cosh sinh M}\emph{\ref{it:cosh+sinh}} we also have $\|\cosh H\|^2 + \|\sinh H\|^2 = \|e^H\|^2$, which implies $\|e^H\| = \|e^{-H}\|$.
	\emph{(\ref{it:r23})} Now $\tilde{H}=-H$, so that $4\|\cosh H\|^2 = 2 + \inner{e^{2H}+e^{-2H}}_0 = 2 + 2\inner{\cosh(2H)}_0$ and $4\|\sinh H\|^2 = 2 - 2\inner{\cosh(2H)}_0$, while \emph{\ref{it:grades cosh sinh}} and \Cref{pr:properties cosh sinh M}\emph{\ref{it:cosh+sinh}} imply $\inner{\cosh(2H)}_0 = \inner{e^{2H}}_0$. Finally, $\|e^H\|^2 = e^{\tilde{H}}*e^H = \inner{e^{-H}e^H}_0 = 1$.
\end{proof}

Note that \emph{(\ref{it:grades cosh sinh})}  and \Cref{pr:properties cosh sinh M}\emph{\ref{it:cosh+sinh}} restrict which grades $e^H$ can include, and when $r \neq 0$ its components are divided between $\cosh H$ and $\sinh H$ according to their grades.

\end{document}